\documentclass[onefignum,onetabnum]{siamart220329}

\usepackage{amsfonts}
\usepackage{graphicx}
\usepackage[caption=false]{subfig}
\usepackage{epstopdf}
\ifpdf
  \DeclareGraphicsExtensions{.eps,.pdf,.png,.jpg}
\else
  \DeclareGraphicsExtensions{.eps}
\fi

\usepackage{mathtools}
\usepackage{mathrsfs}
\usepackage{amsmath}
\usepackage{amsfonts}
\usepackage{amssymb}
\usepackage{enumitem}
\usepackage{color}

\setenumerate{label=\textup{(\alph*)}}

\newsiamremark{remark}{Remark}
\newsiamremark{assumption}{Assumption}
\Crefname{subsection}{Section}{Sections}
\crefname{subsection}{section}{sections}

\Crefrangeformat{enumi}{Assumptions #3#1#4--#5#2#6}
\crefname{corollary}{Corollary}{Corollaries}
\Crefname{assumption}{Assumption}{Assumptions}
\Crefrangeformat{assumption}{Assumptions #3#1#4--#5#2#6}
\crefname{theorem}{Theorem}{Theorems}
\crefname{lemma}{Lemma}{Lemmas}
\crefname{proposition}{Proposition}{Propositions}
\crefname{enumi}{Assumption}{Assumptions}
\newlist{enumthm}{enumerate}{1}
\setlist[enumthm]{label=\textup{(\alph*)},ref=\theassumption~\textup{(\alph*)}}
\crefalias{enumthmi}{assumption}

\headers{Sample Size Estimates for PDE-Constrained
	Optimization}{Johannes Milz and Michael Ulbrich}

\newcommand{\ourtitle}{Sample Size Estimates for Risk-Neutral Semilinear PDE-Constrained Optimization}
\title{\ourtitle\thanks{August 23, 2023}}

\author{Johannes Milz\thanks{H.\ Milton Stewart School of Industrial and Systems Engineering, Georgia Institute of Technology, Atlanta, Georgia 30332, 
		USA, \email{johannes.milz@isye.gatech.edu}}
\and  Michael Ulbrich\thanks{Technical University of Munich, Department of 
	Mathematics,  Boltzmannstr.\ 3, 85748 Garching b.\ M\"unchen,
	Germany, 
	\email{mulbrich@ma.tum.de}}}

\usepackage{amsopn}

\newcommand{\functionArgument}[1]{\ifthenelse{\equal{#1}{}}  
	{}
	{({#1})}
}

\newcommand{\dsp}{X}
\newcommand{\dualpHzeroone}[3][]
{\langle #2, #3 \rangle_{H^{-1}(\domain), H_0^1(\domain)}}
\newcommand{\spL}[2]{\mathscr{L}(#1, #2)}
\newcommand{\adhsp}{H_{\text{ad}}}
\newcommand{\rad}{r_{\text{ad}}}
\newcommand{\wprox}[3][]{\mathrm{prox}_{#2}\functionArgument{#3}}
\newcommand{\prox}[2]{\wprox[]{#1}{#2}}
\newcommand{\vadcsp}{V_{\text{ad}}}
\newcommand{\norm}[2][2]{\|#2\|_{#1}}
\newcommand{\radius}{R_{\text{ad}}}

\newcommand{\rrhs}{b}

\newcommand{\inner}[3][]{( #2, #3 )_{#1}}

\newcommand{\pobj}{J}
\newcommand{\rpobj}{\pobj}
\newcommand{\erpobj}{F}
\newcommand{\csp}{U}
\newcommand{\adcsp}{\csp_\text{ad}}
\newcommand{\ssp}{Y}

\newcommand{\hsp}{H}
\newcommand{\domain}{D}
\newcommand{\lb}{\mathfrak{l}}
\newcommand{\ub}{\mathfrak{u}}
\newcommand{\rv}{Z}
\newcommand{\rvv}{W}
\DeclareMathOperator*{\argmin}{arg\,min}
\newcommand{\naturals}{\mathbb{N}}
\newcommand{\real}{\mathbb{R}}
\newcommand{\embedding}{\xhookrightarrow{}}

\newcommand{\eu}{\ensuremath{\mathrm{e}}}
\newcommand{\du}{\ensuremath{\mathrm{d}}}
\newcommand{\wpone}{w.p.~$1$}
\newcommand{\tand}{\text{and}}
\newcommand{\tif}{\text{if}}
\newcommand{\tfa}{\text{for all}}
\newcommand{\frechet}{Fr\'echet}
\newcommand{\Caratheodory}{Carath\'eodory}
\newcommand{\hoelder}{H\"older}
\newcommand{\friedrichs}{Friedrichs}
\newcommand{\cnop}{Q}
\newcommand{\nop}{q}

\DeclarePairedDelimiterXPP\cE[1]{\mathbb{E}}[]{}{%
	
	#1}

\newcommand{\MonteArg}[1]{\ifthenelse{\equal{#1}{}}  
	{}
	{[{#1}]}
}

\newcommand{\cA}{\mathcal{A}}
\newcommand{\cF}{\mathcal{F}}

\DeclarePairedDelimiterXPP\Prob[1]{\mathrm{Prob}}(){}{
	
	#1}

\ifpdf
\hypersetup{
  pdftitle={\ourtitle},
  pdfauthor={Johannes Milz and Michael Ulbrich}
}
\fi

\begin{document}

\maketitle

\begin{abstract}
  The sample average approximation (SAA) 
  approach is applied to risk-neutral optimization problems
  governed by semilinear elliptic partial differential equations
  with random inputs.
  After constructing a compact set that contains
  the  SAA critical points, we
  derive nonasymptotic sample size estimates
  for  SAA critical points
  using the covering number approach.
  Thereby, we derive upper bounds on the number of
  samples needed to obtain accurate critical points
  of the risk-neutral PDE-constrained optimization problem
  through SAA critical points.
  We quantify 
  accuracy using expectation and exponential tail bounds.
  Numerical illustrations are presented.
\end{abstract}



\section{Introduction}
\label{sec:intro}
Many objective functions of  stochastic programs involve
expectations of parameterized objective functions in one way or another. 
Large-scale stochastic programs arise in a multitude of applications, such as
machine learning \cite{Lan2020,Shalev-Shwartz2010}, 
feedback stabilization of autonomous systems 
\cite{Kunisch2020}, statistical estimation 
\cite{Huber1967,Royset2019}, and
optimization of  differential equations under uncertainty
\cite{Phelps2016,Wechsung2021}.
When the parameter space is 
high-dimensional, these expectations cannot be
accurately evaluated. A common approach to approximate 
such stochastic programs is  the sample average approximation (SAA) method,
yielding the SAA problem \cite{Kleywegt2002,Shapiro2003,Shapiro2005}. 
The SAA objective function is defined by the sample average of the
parameterized objective function.
The SAA approach has become popular in  the literature on optimization
under uncertainty with partial differential equations (PDEs)
\cite{Hoffhues2020,Roemisch2021,Wechsung2021}. 
The key characteristic of the SAA approach 
as applied to risk-neutral PDE-constrained
optimization is 
the fact that it yields PDE-constrained optimization problems
which can be solved efficiently 
using existing, rapidly converging algorithms.
While the SAA approach is 
easy to use, a central question is: how many samples are sufficient
to obtain approximate critical points for the stochastic program
via SAA critical points with high probability?
An answer to this question also addresses the computational complexity 
of risk-neutral PDE-constrained
optimization problems, as 
potentially expensive simulations of PDEs  are required
for each sample.
In the present manuscript, we answer this question for a class
of risk-neutral semilinear PDE-constrained optimization problems.
Our analysis is inspired by those in 
\cite{Shapiro2003,Shapiro2005,Shapiro2021}.

We derive nonasymptotic sample size estimates
for  SAA critical points of infinite dimensional
optimization problems governed by semilinear PDEs with random inputs. 
We consider the risk-neutral 
PDE-constrained optimization problem
\begin{align}
\label{eq:ocp}
\min_{u \in L^2(\domain)}\, 
(1/2) \cE{\norm[L^2(\domain)]{S(u, \xi)-y_d}^2}
+ (\alpha/2) \norm[L^2(\domain)]{u}^2
+ \psi(u),
\end{align}
where $\alpha > 0$, 
$\domain \subset \real^d$ is a bounded Lipschitz
domain, $y_d \in L^2(\domain)$, 
$L^1(\domain)$ $(L^2(\domain))$
is the Lebesgue space of (square) integrable functions
defined on $\domain$, 
and $\xi$ is a random element mapping from a
complete probability space
to a complete probability space with 
sample space $\Xi$ being a complete, separable metric space.
Moreover,
for each $(u,\xi) \in L^2(\domain) \times \Xi$,
$y = S(u,\xi) \in  H_0^1(\domain)$ solves the state equation
\begin{align}
\label{eq:Feb0320211603}
A(\xi) y + \cnop(y) = \rrhs(\xi) + B(\xi)u.
\end{align}
The Sobolev space $H_0^1(\domain)$ is
formally introduced in \cref{sec:notation}.
We state assumptions on  the semilinear PDE
\eqref{eq:Feb0320211603} in \cref{sec:rnpdeopt}. 
The function $\psi : L^2(\domain) \to (-\infty,\infty]$ in \eqref{eq:ocp}
is proper, convex and lower semicontinuous.  An example for 
$\psi$ is given by $\psi(u) = \gamma \norm[L^1(\domain)]{u}$
if $u \in \csp_0$
and $\psi(u) =  \infty$ otherwise,
where  $\gamma \geq 0$ and
$\csp_0 = \{\, u \in L^2(\domain) \colon \, \lb \leq u \leq 
\ub
\,\}$
with $\lb$, $\ub \in L^2(\domain)$
and $\lb \leq \ub$. This function and its variants arise in 
control device placement applications such as
tidal turbine layout optimization \cite{Funke2016}.
The SAA problem of \eqref{eq:ocp} is given by
\begin{align}
\label{eq:saa}
\min_{u \in L^2(\domain)}\, 
\frac{1}{2N} \sum_{i=1}^N\norm[L^2(\domain)]{S(u, \xi^i)-y_d}^2
+ (\alpha/2) \norm[L^2(\domain)]{u}^2
+ \psi(u),
\end{align}
where	$\xi^1$, $\xi^2, \ldots$ 
are independent identically distributed $\Xi$-valued 
random elements defined on a common complete probability
space $(\Omega, \cF, P)$
and  each $\xi^i$ has the  same distribution
as  $\xi$. 
Here, $N \in \naturals$ is the sample size.
Critical points of the SAA problem 
\eqref{eq:saa} can be efficiently computed using semismooth
Newton methods \cite{Ulbrich2011,Stadler2009,Mannel2020}.

The feasible set, 
$\{\, u \in L^2(\domain) \colon\, \psi(u) < \infty \,\}$,
is generally non-compact.
The lack of compactness
and the nonconvexity of  
the risk-neutral PDE-constrained optimization problem \eqref{eq:ocp} 
and its SAA problems 
complicate the derivation of sample size estimates,
as these are typically derived using covering numbers of the feasible set
\cite{Cucker2002,Kankova1978,Shapiro2003,Shapiro2005,Royset2019}.
Sample size estimates can be established for convex stochastic
progams without the covering number approach
\cite{Guigues2017,Shalev-Shwartz2010}. However, 
the risk-neutral PDE-constrained optimization problem \eqref{eq:ocp} 
and its SAA problems are  nonconvex.
For deriving sample size estimates,
our initial observation is that critical points of the SAA problem
\eqref{eq:saa} are contained in
a compact subset of the feasible set.
To construct the compact set, we use
optimality conditions, PDE stability estimates,
and higher regularity
of the reduced parameterized objective function's gradient. 
Utilizing the covering numbers 
for Sobolev function classes established in
\cite{Birman1967,Birman1980},
we derive sample size estimates using arguments similar to those 
in
\cite{Shapiro2003,Shapiro2005,%
	Shapiro2008,Shapiro2021,Royset2019,Cucker2002,Kankova1978,Mei2018}. 
However, our
focus is the analysis of SAA critical points as opposed to SAA solutions, as
risk-neutral semilinear PDE-constrained optimization problems are 
nonconvex.
For classes of risk-neutral linear elliptic PDE-constrained optimizations, 
SAA solutions are analyzed in
\cite{Hoffhues2020,Martin2021,Milz2022c,Milz2021,Roemisch2021}.
However, the analysis
in \cite{Milz2021} does not generalize
to nonconvex problems. 

Alternative approximation approaches 
for risk-neutral PDE-constrained optimization problems
include, for example, quasi-Monte Carlo sampling \cite{Guth2019}, 
low-rank tensor methods \cite{Benner2020,Garreis2017}, and
stochastic collocation \cite{Borzi2011,Kouri2013,Tiesler2012}.
	While we are unaware of a systematic theoretical and/or empirical study
	comparing, for example, Monte Carlo approaches, 
	quasi-Monte Carlo sampling techniques, and sparse grid schemes
	as applied to risk-neutral nonconvex 
	PDE-constrained optimization, we provide a brief comparision of these
	approaches in terms of characteristics used for their theoretical
	convergence anayses.
	Sparse grid-based discretizations may result in nonconvex
	optimization problems  and their analysis requires smoothness properties
	with respect to the parameters and independent random variables
	with Lebesgue densities, but the methods in
        \cite{Kouri2014,Kouri2013,Zahr2019}, which combine sparse grid
	techniques with trust-region schemes, perform very well.
        We refer the reader to 
	\cite{Chen2015} for the analysis of sparse grid-based approximations
	of finite dimensional stochastic programs.
	Monte Carlo sample-based approximations  preserve
	convexity and only require mild integrability properties 
	with respect to possibly infinitely many parameters, 
	but are deemed to be quite slowly convergent.  
	For finite dimensional stochastic programs, it is known
	that quasi-Monte Carlo methods can outperform Monte Carlo sampling
	techniques provided that the integrands satisfy certain 
	regularity conditions \cite[pp.\ 185--189]{Shapiro2021}.
	The quasi-Monte Carlo approach is analyzed in 
	\cite{Guth2019} as applied to risk-neutral 
	linear-quadratic control problems with PDEs, an important
	class of risk-neutral convex PDE-constrained optimization problems.
	As with every numerical scheme, each approximation approach
	for risk-neutral PDE-constrained optimization has its advantages
	and disadvantages. The Monte Carlo sample-based 
	scheme is arguably the most widely used approach to approximating
	stochastic programs. This and the fact that
	it yields standard PDE-constrained  optimization problems build our
	main motivation to establishing sample size estimates.

\section*{Outline of the paper}
We discuss preliminaries and define further notation in
\cref{sec:notation}. A class of risk-neutral semilinear 
PDE-constrained optimization problems is introduced in 
\cref{sec:rnpdeopt}.
PDE stability estimates are derived in
\cref{subsect:semilinearpdes}. In \cref{subsect:derivative}, 
we compute the derivative of the expectation function
and demonstrate the Lipschitz continuity of its gradient in 
\cref{subsec:lipschitz}. The focus of our results in 
\cref{sec:rnpdeopt} is to make the PDE stability estimates'
dependence  on problem-dependent parameters explicit.
\Cref{subsect:existence} discusses
the existence of solutions. 
We combine the results established in
\cref{sec:rnpdeopt} to construct a compact set containing the SAA critical
points in \cref{sec:compatset}. The covering numbers of Sobolev function
classes used to establish nonasymptotic finite sample size estimates
in \cref{sec:samplesizeestimates} are provided in \cref{sec:covering}. 
In \cref{sec:discussion}, we summarize our contributions, put
our approach into perspective, and comment on potential improvements
of our sample size estimates with respect to 
the dimension  of the computational domain.
In the appendices, we discuss auxiliary
results used in \cref{sec:samplesizeestimates} to derive the sample size
estimates. \Cref{sec:setinclusion} discusses the measurability of set
inclusions and \cref{sec:subgaussianbounds} provides sub-Gaussian-type 
expectation and tail bounds of maxima of Hilbert space-valued random vectors.
The  bounds are used in \cref{sect:uniformtailbounds} to derive
uniform expectation and 
exponential tail bounds for normal maps in Hilbert spaces.
These bounds are applied in \cref{sec:samplesizeestimates}
for obtaining our nonasymptotic sample size estimates
for SAA critical points.

\section{Preliminaries and further notation}
\label{sec:notation}

Relationships between random variables, vectors, 
and elements hold with
probability one (\wpone) if not specified otherwise.
Metric spaces are equipped with their Borel $\sigma$-field.
Let $\dsp$, $\dsp_1$ and $\dsp_2$ be real Banach spaces. 
The space of bounded, linear operators from $\dsp_1$
to $\dsp_2$ is denoted by
$\spL{\dsp_1}{\dsp_2}$. We define
$\dsp^* = \spL{\dsp}{\real}$.
The adjoint operator of 
$\Upsilon \in \spL{\dsp_1}{\dsp_2}$
is denoted by $\Upsilon^* \in \spL{\dsp_2^*}{\dsp_1^*}$.
Throughout the text, 
$(\Theta, \mathcal{A}, \mu)$ is a complete
probability space.
An operator-valued mapping $\Upsilon : \Theta \to \spL{\dsp_1}{\dsp_2}$
is called uniformly measurable
if there exists a sequence of simple operators
$\Upsilon_k : \Theta  \to \spL{\dsp_1}{\dsp_2}$
such that $\Upsilon_k(\theta)$
converges to
$\Upsilon(\theta)$ 
in $\spL{\dsp_1}{\dsp_2}$
as $k \to \infty$ 
for all $\theta \in \Theta$;
cf.\ \cite[Def.\ 3.5.5]{Hille1957}.
A function $\upsilon : \Theta   \to \dsp$ is
strongly measurable if there exists a sequence
of simple functions $\upsilon_k : \Theta  \to \dsp $
such that $\upsilon_k(\theta) \to \upsilon(\theta)$
as $k \to \infty$ for all $\theta \in \Theta$
\cite[Def.\ 1.1.4]{Hytoenen2016}.
If $\dsp$ is separable, then
$\upsilon : \Theta \to \dsp$ is strongly measurable
if and only if it is measurable
\cite[Cor.\ 1.1.2 and Thm.\ 1.1.6]{Hytoenen2016}.
	Let $\Gamma \colon \Theta \rightrightarrows \dsp$
be a set-valued mapping with closed images.
The mapping $\Gamma$ is called  measurable
if $\Gamma^{-1}(V) \in \cA$ for each open set
$V \subset X$ \cite[Def.\ 8.1.1]{Aubin2009}.
Here, $\Gamma^{-1}$ is the inverse image of $\Gamma$.
An operator $\Upsilon  \in \spL{\dsp_1}{\dsp_2}$ is compact
if $\Upsilon(V) \subset \dsp_2$ is precompact for each 
bounded set $V \subset \dsp_1$.
Let $\ssp$ be a complete metric space. 
A map $\Phi : \ssp \times \Theta \to \dsp$
is a \Caratheodory\ function if $\Phi(\cdot, \theta)$ is continuous
for all $\theta \in \Theta$ and $\Phi(y,\cdot)$ is measurable
for all $y \in \ssp$ \cite[p.\ 311]{Aubin2009}.
The \frechet\ derivative of a mapping $h$
with respect to $y$
is denoted by  $h_y$.
For a real Hilbert space $\hsp$, 
$\inner[\hsp]{\cdot}{\cdot}$ is its inner product
and $\norm[\hsp]{h} = \inner[\hsp]{h}{h}^{1/2}$ its norm.
For a convex, lower semicontinuous, 
proper function $\varphi : \hsp \to (-\infty,\infty]$, 
the proximity operator $\prox{\varphi}{}:\hsp \to \hsp$ of $\varphi$ is defined by
(see  \cite[Def.\ 12.23]{Bauschke2011})
\begin{align*}
\prox{\varphi}{v}
= \argmin_{w\in\hsp}\, 
\varphi(w) + (1/2)\norm[\hsp]{v-w}^2.
\end{align*}
Throughout the text, $\domain \subset \real^d$ is a bounded
Lipschitz domain.
The space $L^2(\domain)$ is identified with its dual.
The Sobolev space $H^1(\domain)$ is defined by
all $L^2(\domain)$-functions with square integrable weak derivatives
and $H_0^1(\domain)$ consists all
$v \in H^1(\domain)$ with zero boundary traces.
We define $H^{-1}(\domain) = H_0^1(\domain)^*$.
The dual pairing 
between $H^{-1}(\domain)$ and $H_0^1(\domain)$ is denoted by
$\dualpHzeroone{\cdot}{\cdot}$.
We equip $H_0^1(\domain)$
with 
$\norm[H_0^1(\domain)]{v} = \norm[L^2(\domain)]{\norm[2]{\nabla v}}$
and  $H^1(\domain)$ with 
$\norm[H^1(\domain)]{v} = (\norm[L^2(\domain)]{v}^2 +
\norm[H_0^1(\domain)]{v}^2)^{1/2}$.
Here,  $\nabla v$ is the weak gradient of $v$
and $\norm[2]{\cdot}$ is the Euclidean norm on $\real^d$.
Throughout the text,  $\iota : H_0^1(\domain) \to L^2(\domain)$
given by $\iota v  = v$ is the embedding operator of 
the compact embedding $H_0^1(\domain) \embedding L^2(\domain)$
and  
$C_\domain$ is \friedrichs' constant, the operator norm of $\iota$.
We have 	$C_\domain =
\norm[\spL{L^2(\domain)}{H^{-1}(\domain)}]{\iota^*}$.
We denote by $|\domain|$ the Lebesgue measure 
of $\domain$ and by $\bar \domain$ its closure.
We define $C^{0,1}(\bar{\domain})$ as the space of
Lipschitz continuous real-valued functions defined on
$\bar{\domain}$ and equip it with the  usual norm
$\norm[C^{0,1}(\bar \domain)]{\cdot}$
defined in \cite[p.\ 16]{Hinze2009}.

Let $(\dsp, \norm[\dsp]{\cdot})$  be a 
normed space, let $\dsp_0 \subset \dsp$ 
be nonempty and totally bounded, and let $\nu > 0$. 
The $\nu$-covering number 
$\mathcal{N}(\nu; \dsp_0, \norm[\dsp]{\cdot})$
is the minimal number of closed
$\norm[\dsp]{\cdot}$-balls  
with radius $\nu$ in $\dsp$ needed to cover 
$\dsp_0$  
(cf.\ \cite[pp.\ 87--88]{Tikhomirov1993}).
This notion of covering
numbers does not require the centers of the $\norm[\dsp]{\cdot}$-balls
be contained in $\dsp_0$.

\section{Risk-neutral semilinear PDE-constrained optimization}
\label{sec:rnpdeopt}

We impose conditions on the parameterized 
PDE \eqref{eq:Feb0320211603} which
ensure the existence and uniqueness of solutions.
These conditions ensure that the semilinear PDE
\eqref{eq:Feb0320211603} is a 
monotone operator equation and hence the existence
of solutions and the stability estimates 
can be established using the Minty--Browder theorem
\cite[Thm.\ 26.A]{Zeidler1990}, for example.
Semilinear PDE-constrained optimization problems
are analyzed, for example, in 
\cite{Garreis2019a,Geiersbach2020,Kouri2020,Troeltzsch2010,Ulbrich2011}.
In this section, our contributions are
to address some measurability questions, 
and to derive stability estimates and a bound on an
objective function's Lipschitz constant 
with the dependence on problem data's characteristics made explicit. 
Even though our derivations are built on standard techniques,
the stability estimates are needed for establishing our 
nonasymptotic sample size estimates (see \cref{sec:samplesizeestimates}).

We define $\erpobj : L^2(\domain) \to \real$
and  $\hat{\erpobj}_N : L^2(\domain) \to \real$ by
\begin{align}
\label{eq:efuns}
\erpobj(u) = (1/2) \cE{\norm[L^2(\domain)]{\iota S(u, \xi)-y_d}^2}
\quad \tand \quad 
\hat{\erpobj}_N(u) = \frac{1}{2N} \sum_{i=1}^N 
\norm[L^2(\domain)]{\iota S(u, \xi^i)-y_d}^2.
\end{align}
As opposed to the problem formulations in \eqref{eq:ocp}
and \eqref{eq:saa},
we make here the use of the embedding operator $\iota$ explicit.
Since the random elements $\xi^1, \xi^2, \ldots$ are defined on
the complete probability space $(\Omega, \cF, P)$, we can view
$\hat{\erpobj}_N$ as a function defined on 
$L^2(\domain) \times \Omega$.
However, we often omit the second argument.

We define the feasible set 
$\adcsp = \{\, u \in L^2(\domain) \colon\, \psi(u) < \infty \,\}$.
\begin{assumption}[{Control regularization and feasible set}]
	\label{ass:adcspbounded}
	\begin{enumthm}[nosep,leftmargin=*]
		\item 
		\label{itm:psi}
		The function 
		$\psi : L^2(\domain) \to (-\infty,\infty]$
		is proper, convex and lower semicontinuous.
		\item 
		\label{itm:adcspbounded}
		For some $\rad \in (0,\infty)$, 
		$\norm[L^2(\domain)]{u} \leq \rad$
		for all $u \in \adcsp$.
	\end{enumthm}
\end{assumption}

\subsection{Semilinear PDEs with random inputs}
\label{subsect:semilinearpdes}

We impose conditions on the data defining the semilinear PDE
\eqref{eq:Feb0320211603}
based on  those used in 
\cite[Assumptions (3.1)--(3.3)]{Kouri2020},
\cite[sect.\ 9.1]{Ulbrich2011}, and
\cite[Assumption 2.2]{Garreis2019a}.

\begin{assumption}[{Semilinear PDE: Problem data}]
	\label{assumption:pde}
	\begin{enumthm}[nosep,leftmargin=*]
		\item 
		\label{itm:domain}
		$\domain \subset \real^d$ is a bounded Lipschitz domain
		with $d \in \{2,3\}$.
		\item 
		\label{itm:A}
		$A :  \Xi \to \spL{H_0^1(\domain)}{H^{-1}(\domain)}$
		is uniformly measurable.
		There exists a constant $\kappa_{\min} > 0$
		such that  $A(\xi)$ is self-adjoint
		and 
		$\dualpHzeroone{A(\xi)y}{y} \geq 
		\kappa_{\min}\norm[H_0^1(\domain)]{y}^2$
		for all $y \in H_0^1(\domain)$
		and each $\xi \in \Xi$.
		
		\item 
		\label{itm:rrhs}
		$\rrhs : \Xi \to H^{-1}(\domain)$
		and 
		$g   : \Xi \to  C^{0, 1}(\bar \domain)$
		are strongly measurable
		and there exist constants $\rrhs_{\max}$, $g_{\max} > 0$
		such that
		$\norm[H^{-1}(\domain)]{\rrhs(\xi)} \leq \rrhs_{\max}$
		and 
		$\norm[C^{0, 1}(\bar \domain)]{g(\xi)} \leq g_{\max}$
		for each $\xi \in \Xi$.
		
		\item 
		\label{itm:B}
		For each  $\xi \in \Xi$, 
		$B(\xi)  \in \spL{L^2(\domain)}{H^{-1}(\domain)}$
		is defined by
		\begin{align}
		\label{eq:B}
		\dualpHzeroone{B(\xi)u}{v} = \inner[L^2(\domain)]{g(\xi)u}{v}.
		\end{align}
		
		\item 
		\label{itm:nonlinearity}
		We define $\cnop : H_0^1(\domain) \to H^{-1}(\domain)$
		by
		$
		\dualpHzeroone{\cnop(y)}{v} 
		= \inner[L^2(\domain)]{\nop(y)}{v}
		$.
		The function $\nop: \real \to \real$ is nondecreasing, 
		twice continuously differentiable, $\nop(0) = 0$, 
		and 
		$|\nop''(t)| \leq c_{\nop} + d_{\nop} |t|^{p-3}$
		for all $t \in \real$,
		where $c_{\nop}$, $d_{\nop} \geq 0$
		and $p \in (3,\infty)$ if $d = 2$
		and $p \in (3, 6]$ if $d = 3$.
	\end{enumthm}
\end{assumption}

\Cref{itm:domain,itm:nonlinearity} ensure that the embedding
$H_0^1(\domain) \embedding L^p(\domain)$ is continuous
\cite[Thm.\ 1.14]{Hinze2009}
and that $\nop$ is twice continuously differentiable 
as a superposition operator from $L^p(\domain)$
to $L^{p^*}(\domain)$ \cite[p.\ 202]{Ulbrich2011}, where
$p^* \in [1,\infty)$ fulfills $1/p + 1/p^* = 1$.
Hence, $\cnop$ is  twice continuously differentiable 
from $H_0^1(\domain)$ to $H^{-1}(\domain)$.
Moreover $\cnop$ is monotone.
Hence \Cref{assumption:pde} ensures that
for each $\xi \in \Xi$, the operator
$H_0^1(\domain) \ni y \mapsto A(\xi) y + \cnop(y) \in H^{-1}(\domain)$
is  continuous and strongly monotone with parameter $\kappa_{\min}$
and hence coercive \cite[p.\ 501]{Zeidler1990}.
We denote by $c_p > 0$  the embedding constant of the embedding 
$H_0^1(\domain) \embedding L^p(\domain)$. 
Under \Cref{itm:domain,itm:nonlinearity}, 
the constant $c_p$ is finite since $d \in \{2,3\}$ and $p \in [1,6]$
\cite[Thm.\ 1.14]{Hinze2009}.
\Cref{lem:properties_B} establishes basic properties
of the mapping $B$. In particular, we 
establish its uniform measurability and some of its compactness properties.
\begin{lemma}
	\label{lem:properties_B}
	If \Cref{itm:domain,itm:rrhs,itm:B}  hold, then
	for all $\xi \in \Xi$, 
	\begin{enumerate}[nosep,leftmargin=*]
		\item 
		\label{itm:Bbounded}
		$B : \Xi \to \spL{L^2(\domain)}{H^{-1}(\domain)}$
		is uniformly measurable
		and 	
		$$\norm[\spL{L^2(\domain)}{H^{-1}(\domain)}]{B(\xi)} 
		\leq C_\domain\norm[C^{0,1}(\bar{\domain})]{g(\xi)},$$ 
		\item 
		\label{itm:Badjointcompact}
		$B(\xi)^*v = \iota[g(\xi)v] = g(\xi) v$
		for all $v \in H_0^1(\domain)$, 
		and $B(\xi)^*$ and $B(\xi)$ are compact.
	\end{enumerate}
\end{lemma}

We prove \Cref{lem:properties_B} using \Cref{lem:Grisvard2011}. 
\Cref{lem:Grisvard2011} establishes an explicit continuity constant
of a certain bilinear mapping in terms of \friedrichs' constant.
\begin{lemma}
	\label{lem:Grisvard2011}
	Let $\domain \subset \real^d$ be a bounded Lipschitz domain.
	If $v \in C^{0,1}(\bar{\domain})$
	and $w \in H_0^1(\domain)$, 
	then $vw\in H_0^1(\domain)$
	and
	$
	\norm[H_0^1(\domain)]{vw} \leq (C_\domain+1)\norm[C^{0,1}(\bar{\domain})]{v}
	\norm[H_0^1(\domain)]{w}
	$.
\end{lemma}
\begin{proof}
	Using \cite[Thm.\ 1.4.1.2]{Grisvard2011}, 
	we have $vw\in H_0^1(\domain)$.
	Since $v$ is Lipschitz continuous, 
	we have
	$\norm[L^\infty(\domain)]{\norm[2]{\nabla v}} 
	\leq \norm[C^{0,1}(\bar{\domain})]{v}$.
	Using $\nabla[vw] = w\nabla v  + v \nabla w$
	and the \hoelder\ inequality, 
	we obtain
	$\norm[H_0^1(\domain)]{vw} \leq 
	\norm[C^{0,1}(\bar{\domain})]{v}(\norm[L^2(\domain)]{w}+
	\norm[L^2(\domain)]{\norm[2]{\nabla w}})
	$.
	Combined
	with \friedrichs' inequality, we obtain the stability estimate.
\end{proof}

\begin{proof}[{Proof of \Cref{lem:properties_B}}]
	Let us fix $u \in L^2(\domain)$, $v \in H_0^1(\domain)$,
	and $\xi \in \Xi$.
	\begin{enumerate}[wide,nosep]
		\item  
		We define $\varrho : C^{0,1}(\bar{\domain}) \to
		\spL{L^2(\domain)}{H^{-1}(\domain)}$
		by
		$\dualpHzeroone{\varrho(g)u}{v} 
		= \inner[L^2(\domain)]{gu}{v}$.
		The map $\varrho$ is Lipschitz continuous
		with Lipschitz constant $C_\domain$.
		Combined with $B = \varrho \circ g$ 
		(see \eqref{eq:B})
		and the strong measurability of $g$,
		we find that
		$B$ is uniformly measurable \cite[Cor.\ 1.1.11]{Hytoenen2016}.
		The bound on the operator norm of $B(\xi)$ is implied by the above
		Lipschitz continuity.
		
		\item 
		We have
		$\inner[L^2(\domain)]{B(\xi)^*v}{u} = \inner[L^2(\domain)]{g(\xi)v}{u}$
		for all $u \in L^2(\domain)$.
		Since $g(\xi) \in C^{0,1}(\bar{\domain})$, 
		$g(\xi)v \in H_0^1(\domain)$
		\cite[Thm.\ 1.4.1.2]{Grisvard2011}.
		Hence, $B(\xi)^*v = \iota[g(\xi)v]$.
		The map
		$H_0^1(\domain) \ni v \mapsto g(\xi)v \in H_0^1(\domain)$
		is linear and bounded
		(see \Cref{lem:Grisvard2011}).
		Hence $B(\xi)^*$ is compact. 
		Now Schauder's theorem
		\cite[Thm.\ 3.4 on p.\ 174]{Conway1985}
		implies that $B(\xi)$ is compact. 
	\end{enumerate}
\end{proof}

The next lemma is essentially known; cf.\
\cite{Ulbrich2011,Kouri2020,Garreis2019a}.

\begin{lemma}
	\label{lem:s_unique_stable}
	Let \Cref{assumption:pde} hold.
	For each $(u,\xi) \in L^2(\domain) \times \Xi$, 
	the parameterized PDE \eqref{eq:Feb0320211603} has a unique
	solution $S(u,\xi) \in H_0^1(\domain)$ and 
	\begin{align}
	\label{eq:Feb0420210958}
	&\norm[H_0^1(\domain)]{S(u,\xi)}
	\leq (1/\kappa_{\min})\norm[H^{-1}(\domain)]{b(\xi)}
	+ ( C_\domain/\kappa_{\min})
	\norm[C^{0,1}(\bar{\domain})]{g(\xi)} \norm[L^2(\domain)]{u}.
	\end{align}
	Moreover, for each $(u_1,u_2,\xi) \in L^2(\domain)^2 \times \Xi$, 
	\begin{align}
	\label{eq:Nov2520211145}
	&\norm[H_0^1(\domain)]{S(u_2,\xi)-S(u_1,\xi)}
	\leq (C_\domain/\kappa_{\min}) \norm[C^{0,1}(\bar{\domain})]{g(\xi)}
	\norm[L^2(\domain)]{u_2-u_1}.
	\end{align}
\end{lemma}
\begin{proof}
	The existence and uniqueness is implied by
	\cite[Thm.\ 26.A]{Zeidler1990}. We have
	$\kappa_{\min}\norm[H_0^1(\domain)]{S(u,\xi)}
	\leq \norm[H^{-1}(\domain)]{b(\xi) +  B(\xi) u}$
	(cf.\ \cite[eq.\ (3.5)]{Kouri2020})
	and
	\begin{align}
	\label{eq:eq37}
	\kappa_{\min}\norm[H_0^1(\domain)]{S(u_2,\xi)-S(u_1,\xi)}
	\leq \norm[H^{-1}(\domain)]{B(\xi)[u_2-u_1]};
	\end{align}
	cf.\ \cite[p.\ 560]{Zeidler1990} and \cite[eq.\ (3.5)]{Kouri2020}.
	Combined with \Cref{lem:properties_B}, we obtain
	\eqref{eq:Feb0420210958} and \eqref{eq:Nov2520211145}.
\end{proof}

\subsection{Derivative computation}
\label{subsect:derivative}

We show that the functions $\erpobj$ and $\hat{\erpobj}_N$ are 
\frechet\ differentiable, compute their derivatives
using the adjoint approach, and derive further stability estimates.
The derivative formulas and the stability estimates are
crucial for obtaining our finite sample size estimates.

Let us define $\rpobj : L^2(\domain) \times \Xi \to \real$
by 
\begin{align}
\label{eq:rpobj}
\rpobj(u,\xi) = (1/2)\norm[L^2(\domain)]{\iota S(u,\xi)-y_d}^2.
\end{align}
For each $(u,\xi) \in L^2(\domain) \times \Xi$, 
$z=z(u,\xi) \in H_0^1(\domain)$ solves the adjoint equation
\begin{align}
\label{eq:adjoint}
[A(\xi)  + \cnop_y(S(u,\xi))]z  = - \iota^*[\iota S(u,\xi)-y_d].
\end{align}
To derive  \eqref{eq:adjoint}, we have used the
fact that $A(\xi)$
(see \Cref{itm:A}) and $\cnop_y(y)$ are self-adjoint
for each $(y,\xi) \in H_0^1(\domain) \times \Xi$
(cf.\ \cite[eq.\ (2.8)]{Garreis2019a}).

\begin{proposition}
	\label{prop:efunsdifferentiable}
	If \Cref{assumption:pde} holds, 
	then $\erpobj$ and $\hat{\erpobj}_N$ defined in \eqref{eq:efuns} 
	are \frechet\ differentiable.
	For each $u \in L^2(\domain)$, it holds that
	$\nabla \erpobj(u)$, $\nabla \hat{\erpobj}_N(u) \in H_0^1(\domain)$,
	\begin{align}
	\label{eq:Feb0820212103}
	\nabla \erpobj(u) = 
	-\cE{g(\xi) z(u, \xi)}
	\quad \tand \quad 
	\nabla \hat{\erpobj}_N(u) = 
	-\frac{1}{N} \sum_{i=1}^N g(\xi^i) z(u, \xi^i).
	\end{align}
\end{proposition}%

The proof of \Cref{prop:efunsdifferentiable} uses
\Cref{lem:Feb0720211501,lem:Nov252021204}.

\begin{lemma}
	\label{lem:Feb0720211501}
	If \Cref{assumption:pde} holds, then the following statements
	hold.
	\begin{enumerate}[nosep,leftmargin=*]
		\item For each $(u,\xi) \in L^2(\domain) \times \Xi$, 
		the adjoint equation \eqref{eq:adjoint} has a unique
		solution $z(u,\xi) \in H_0^1(\domain)$ and 
		\begin{align}
		\label{eq:Feb0720211501}
		\begin{aligned}
		\norm[H_0^1(\domain)]{z(u,\xi)}
		& \leq (C_\domain/\kappa_{\min})\norm[L^2(\domain)]{\iota S(u,\xi)-y_d}.
		\end{aligned}
		\end{align}
		\item 
		\label{itm:s_caratheodory}
		The solution operator $S : L^2(\domain) \times \Xi \to H_0^1(\domain)$
		is a \Caratheodory\ mapping.
		\item 	
		\label{itm:psiystrongly}
		If $s : \Xi  \to H_0^1(\domain)$ 
		is measurable, 
		then 
		$\Xi \ni \xi \mapsto \cnop_y(s(\xi)) 
		\in \spL{H_0^1(\domain)}{H^{-1}(\domain)}$
		is uniformly measurable.
		\item 
		\label{itm:zcaratheodory}
		The adjoint state 
		$z :  L^2(\domain) \times \Xi \to H_0^1(\domain)$
		is a \Caratheodory\ function.
		\item $L^2(\domain) \times \Xi \ni 
		(u,\xi) \mapsto g(\xi) z(u,\xi) \in H_0^1(\domain)$
		is a \Caratheodory\ map.
		For each $u \in L^2(\domain)$, 
		$\Xi \ni \xi \mapsto g(\xi)z(u,\xi) \in H_0^1(\domain)$
		is Bochner integrable
		and
		\begin{align}
		\label{eq:23Nov20211913}
		\norm[H_0^1(\domain)]{g(\xi)z(u,\xi)}
		&\leq (C_D+1)\norm[C^{0,1}(\bar \domain)]{g(\xi)}
		\norm[H_0^1(\domain)]{z(u,\xi)}.
		\end{align}
	\end{enumerate}
\end{lemma}
\begin{proof}
	\begin{enumerate}[wide,nosep]
		\item The assertions are a consequence of the Lax--Milgram lemma
		and the fact that
		$C_\domain =
		\norm[\spL{L^2(\domain)}{H^{-1}(\domain)}]{\iota^*}$.
		\item Let us define the nonlinear operator
		$E : H_0^1(\domain) \times L^2(\domain) \times \Xi  \to H^{-1}(\domain)$
		by
		$E(y,u,\xi) = A(\xi) y + \cnop(y) - \rrhs(\xi) - B(\xi)u$.
		Then $S(u,\xi)$ solves 
		\eqref{eq:Feb0320211603} if and only if $E(S(u,\xi),u,\xi) = 0$.
		As in the proofs of 
		\cite[Lems.\ 9.2 and 9.6]{Ulbrich2011}, we can show that
		$E(\cdot,\cdot,\xi)$ is twice continuously differentiable
		for each $\xi \in \Xi$ 	and that
		$\dualpHzeroone{E_y(y,u,\xi)v}{v} \geq
		\kappa_{\min}\norm[H_0^1(\domain)]{v}^2$
		for each $(y,v,u,\xi) \in H_0^1(\domain)^2 \times L^2(\domain) \times \Xi$. 
		Hence, the implicit function theorem 
		and the  Lax--Milgram lemma \cite[Lem.\ 1.8]{Hinze2009}
		ensure that $S(\cdot, \xi)$ is twice continuously
		differentiable for each $\xi \in \Xi$.
		For each $u \in L^2(\domain)$,
		the measurability of $S(u,\cdot)$ is implied by
		the proof of \cite[Thm.\ 3.12]{Garreis2019}.
		\item Since $H_0^1(\domain) \ni y \mapsto \cnop_y(y) \in 
		\spL{H_0^1(\domain)}{H^{-1}(\domain)}$ is continuous and
		$s$ is strongly measurable, 
		$\cnop_y \circ s$ is uniformly measurable
		\cite[Cor.\ 1.1.11]{Hytoenen2016}.
		\item 
		For each $y \in H_0^1(\domain)$, 
		$v \mapsto \cnop_y(y)v$ is monotone. 
		Hence, 		for each $\xi \in \Xi$, 
		$\dualpHzeroone{A(\xi)y +  \cnop_y(S(u,\xi))y}{y} \geq 
		\kappa_{\min}\norm[H_0^1(\domain)]{y}^2$
		for all $y \in H_0^1(\domain)$.
		Now, the implicit function theorem
		and part~\ref{itm:s_caratheodory} can be used to deduce
		the twice continuous differentiability  of $z(\cdot, \xi)$
		for all $\xi \in \Xi$.
		Fix $u \in L^2(\domain)$.
		\Cref{assumption:pde}, and
		parts~\ref{itm:s_caratheodory} and \ref{itm:psiystrongly} 
		ensure 
		that $\Xi \ni \xi \mapsto A(\xi) + \cnop_y(S(u,\xi))$
		is uniformly measurable.
		The measurability of $z(u,\cdot)$ is now implied by
		part~\ref{itm:s_caratheodory}
		when applied to the adjoint  equation \eqref{eq:adjoint}
		rather than the state equation \eqref{eq:Feb0320211603}.

		\item For each
		$\xi \in \Xi$, 
		part~\ref{itm:zcaratheodory}
		and \Cref{lem:Grisvard2011} ensure the continuity of
		$L^2(\domain) \ni  u \mapsto g(\xi) z(u,\xi) \in H_0^1(\domain)$.
		Now, let $u \in L^2(\domain)$.
		Let us define 
		$\varrho_1 : \Xi \to C^{0,1}(\bar \domain) \times H_0^1(\domain)$
		by $\varrho_1(\xi) = (g(\xi), z(u,\xi))$.
		Since $g$ and $z(u,\cdot)$ are strongly measurable
		(see part~\ref{itm:zcaratheodory}),
		$\varrho_1$ is strongly measurable.
		We also define $\varrho_2 : C^{0,1}(\bar \domain) \times H_0^1(\domain)
		\to H_0^1(\domain)$
		by $\varrho_2(g,z) = g z$.
		\Cref{lem:Grisvard2011} ensures that
		 $\varrho_2$ is bounded.
		Combined with the fact that $\varrho_2$  is bilinear, 
		we find that  $\varrho_2$ is continuous.
		We have $g(\xi) z(u,\xi) = (\varrho_2 \circ \varrho_1)(\xi)$
		for all $\xi \in \Xi$.
		Thus, the chain rule
		\cite[Cor.\ 1.1.11]{Hytoenen2016} ensures
		the  measurability
		of $\Xi \ni \xi \mapsto g(\xi)z(u,\xi) \in H_0^1(\domain)$.
		
		\Cref{lem:Grisvard2011} yields \eqref{eq:23Nov20211913}. 
		Using \eqref{eq:Feb0420210958}, 
		\eqref{eq:Feb0720211501}, \eqref{eq:23Nov20211913}, 
		\Cref{itm:rrhs}, 	and the  measurability
		of $\Xi \ni \xi \mapsto g(\xi)z(u,\xi) \in H_0^1(\domain)$, 
		we have
		$\cE{\norm[H_0^1(\domain)]{g(\xi)z(u,\xi)}} < \infty$.
		Hence, $\Xi \ni \xi \mapsto g(\xi)z(u,\xi) \in H_0^1(\domain)$
		is Bochner integrable \cite[Prop.\ 1.2.2]{Hytoenen2016}.
	\end{enumerate}
\end{proof}

\Cref{lem:Nov252021204} establishes bounds on the gradient of $\rpobj$.
\begin{lemma}
	\label{lem:Nov252021204}
	If \Cref{assumption:pde} holds, then
	$\rpobj(\cdot,\xi)$ is twice continuously differentiable 
	for all $\xi \in \Xi$.
	For all 	$(u,\xi) \in L^2(\domain) \times \Xi$,
	$\nabla_u \rpobj(u,\xi) = -g(\xi)z(u,\xi)$ and
	\begin{align}
		\label{eq:Dec2820211103'}
	\norm[H_0^1(\domain)]{\nabla_u \rpobj(u,\xi)}
	& \leq 
	(C_D+1)(C_\domain/\kappa_{\min})
	\norm[C^{0,1}(\bar \domain)]{g(\xi)}
	\big( 
	(C_\domain/\kappa_{\min})\norm[H^{-1}(\domain)]{b(\xi)}
	+ 
	\\
	\nonumber
	&\quad \quad 
	( C_\domain^2/\kappa_{\min})
	\norm[C^{0,1}(\bar{\domain})]{g(\xi)} \norm[L^2(\domain)]{u}
	+
	\norm[L^2(\domain)]{y_d}\big),
	\\
	\label{eq:Nov252021210}
	\norm[L^2(\domain)]{\nabla_u \rpobj(u,\xi)}	
	& \leq 
	(C_D^2/\kappa_{\min})\norm[C^{0,1}(\bar \domain)]{g(\xi)}
	\big(C_D \norm[H_0^1(\domain)]{S(u,\xi)} + \norm[L^2(\domain)]{y_d}\big).
	\end{align}
\end{lemma}

\begin{proof}
	Adapting the proof of 	\cite[Lem.\ 9.8]{Ulbrich2011}, 
	we obtain that $\rpobj(\cdot, \xi)$
	is twice continuously differentiable on $L^2(\domain)$ with 
	$\nabla_u \rpobj(u,\xi) = - g(\xi) z(u,\xi)$
	for each $(u,\xi) \in L^2(\domain)$.
	Combined with 	\eqref{eq:Feb0720211501} 
	and \eqref{eq:23Nov20211913}, we obtain	
	\begin{align*}
			\norm[H_0^1(\domain)]{\nabla_u \rpobj(u,\xi)}
			\leq  
			(C_D+1)(C_\domain/\kappa_{\min})
			\norm[C^{0,1}(\bar \domain)]{g(\xi)}
			\norm[L^2(\domain)]{\iota S(u,\xi)-y_d}.
	\end{align*}
	Together with \eqref{eq:Feb0420210958}, 
	we obtain \eqref{eq:Dec2820211103'}. 
	Using \eqref{eq:Feb0720211501} and \friedrichs' inequality,
	we have
	\begin{align*}
	\norm[L^2(\domain)]{\nabla_u \rpobj(u,\xi)}
	\leq \norm[L^2(\domain)]{g(\xi) z(u,\xi)}
	\leq \tfrac{C_\domain^2}{\kappa_{\min}}
	\norm[C^{0,1}(\bar \domain)]{g(\xi)}
	\norm[L^2(\domain)]{\iota S(u,\xi)-y_d}.
	\end{align*}
	Hence, we obtain \eqref{eq:Nov252021210}.
\end{proof}

\begin{proof}[{Proof of \Cref{prop:efunsdifferentiable}}]
	Combining the stability estimates
	\eqref{eq:Feb0420210958}
	and  \eqref{eq:Dec2820211103'},
	and 
	\friedrichs' inequality,
	we can apply \cite[Lem.\ C.3]{Geiersbach2020} to deduce 
	the \frechet\ differentiability of 
	$\erpobj$ and of $\hat{\erpobj}_N$
	and the formulas in \eqref{eq:Feb0820212103}. 
	Owing to \eqref{eq:Feb0820212103} and
	\Cref{lem:Feb0720211501}, we find
	that $\nabla \erpobj(u)$, $\nabla \hat{\erpobj}_N(u) \in H_0^1(\domain)$. 
\end{proof}

\subsection{Lipschitz constant computation}
\label{subsec:lipschitz}
We compute a deterministic Lipschitz constant of the gradient
$\nabla_u \rpobj(\cdot,\xi)$ on $\adcsp$ for $\xi \in \Xi$.
\begin{proposition}
	\label{prop:efunsdifferentiable'}
	If \Cref{assumption:pde,itm:adcspbounded} hold, 
	then for each $\xi \in \Xi$, 
	the map
	$\nabla_u \rpobj(\cdot,\xi)$
	is Lipschitz continuous
	on $\adcsp$ with Lipschitz constant $L_{\nabla \rpobj}$ given by
	\begin{align}
	\nonumber 
	L_{\nabla \rpobj}
	& = 
	C_\domain g_{\max}
	\bigg(
	(C_\domain^3/\kappa_{\min}^2)
	g_{\max}
	\\
	\nonumber 
	&\quad +
	\bigg[	(C_\domain/\kappa_{\min}^2)c_p^3
	g_{\max}
	\Big(
	c_{\nop} |\domain|^{(p-3)/p} 
	\\
		\label{eq:nablarpobjLipschitz}
	&\quad\quad +
	d_{\nop} c_p^{p-3}
	\big(
	(1/\kappa_{\min})b_{\max} 
	+
	3(C_\domain/\kappa_{\min})g_{\max}
	\rad
	\big)^{p-3}
	\Big)
	\\
	\nonumber 
	&\quad\quad\quad\cdot 
	\Big(
	(C_\domain^2/\kappa_{\min}^2)b_{\max} 
	+ (C_\domain^3/\kappa_{\min}^2)g_{\max}
	\rad
	+ (C_\domain/\kappa_{\min})\norm[L^2(\domain)]{y_d}
	\Big)
	\bigg]
	\bigg).
	\end{align}
\end{proposition}

\Cref{prop:efunsdifferentiable'} is established 
using \Cref{lem:lem41}.
\begin{lemma}[{see \cite[Lem.\ 5.2]{Garreis2019} and \cite[Lem.\ 4.1]{Garries2019c}}]
	\label{lem:lem41}
	If \Cref{itm:domain,itm:nonlinearity} hold, 
	then for all $y_1$, $y_2 \in H_0^1(\domain)$,
	\begin{align*}
	\begin{aligned}
	&\norm[\spL{H_0^1(\domain)}{H^{-1}(\domain)}]
	{\cnop_y(y_2)-\cnop_y(y_1)}
	\\ 
	&	\leq c_p^3\Big(
	c_{\nop} |\domain|^{(p-3)/p} +
	d_{\nop} c_p^{p-3}
	\big(
	\norm[H_0^1(\domain)]{y_2} + \norm[H_0^1(\domain)]{y_2-y_1}
	\big)^{p-3}
	\Big)
	\norm[H_0^1(\domain)]{y_2-y_1}.
	\end{aligned}
	\end{align*}
\end{lemma}

\begin{proof}[{Proof of \Cref{prop:efunsdifferentiable'}}]
	Fix $(u_1,u_2,\xi) \in L^2(\domain)^2 \times \Xi$.
	First, we show that  
	\begin{align}
	\label{eq:Nov2520211221}
	\begin{aligned}
	& \norm[H_0^1(\domain)]{z(u_2,\xi)-z(u_1,\xi)}
	\leq 
	(C_\domain^3/\kappa_{\min}^2)
	\norm[C^{0,1}(\bar{\domain})]{g(\xi)}
	\norm[L^2(\domain)]{u_2-u_1} 
	\\
	&\quad +
	\bigg[(C_\domain/\kappa_{\min}^2)c_p^3
	\norm[C^{0,1}(\bar{\domain})]{g(\xi)}
	\Big(
	c_{\nop} |\domain|^{(p-3)/p} 
	\\
	&\quad\quad +
	d_{\nop} c_p^{p-3}
	\big(
	\norm[H_0^1(\domain)]{S(u_2,\xi)} + 
	\norm[H_0^1(\domain)]{S(u_2,\xi)-S(u_1,\xi)}
	\big)^{p-3}
	\Big)
	\\
	&\quad\quad\quad\cdot 
	\norm[H_0^1(\domain)]{z(u_1,\xi)}
	\bigg] \norm[L^2(\domain)]{u_2-u_1}.
	\end{aligned}
	\end{align}
	Using  arguments similar to those 
	used to derive the stability estimate in the proof of
	\cite[Prop.\ 4.4]{Kouri2020}, the definition
	of $\rpobj$ (see \eqref{eq:rpobj}), and the triangle inequality, we have
	\begin{align}
	\nonumber 
	\kappa_{\min}\norm[H_0^1(\domain)]{z(u_2,\xi)-z(u_1,\xi)}
	&\leq 
	\norm[H^{-1}(\domain)]{\iota^*\iota[S(u_2,\xi)-S(u_1,\xi)]}
	\\
	\label{eq:adjointlipschitz}
	&\quad +
	\norm[H^{-1}(\domain)]{[\cnop_y(S(u_2,\xi))
		-\cnop_y(S(u_1,\xi))]^*z(u_1,\xi)}.
	\end{align}
	Using \eqref{eq:Nov2520211145}, we have
	\begin{align}
	\label{eq:Dec1520211631}
	\begin{aligned}
	\norm[H^{-1}(\domain)]{\iota^*\iota[S(u_2,\xi)-S(u_1,\xi)]}
	&\leq C_\domain^2
	\norm[H_0^1(\domain)]{S(u_2,\xi)-S(u_1,\xi)}
	\\
	&\leq (C_\domain^3/\kappa_{\min})
	\norm[C^{0,1}(\bar{\domain})]{g(\xi)}
	\norm[L^2(\domain)]{u_2-u_1}.
	\end{aligned}
	\end{align}
	To derive a bound on the second term in 
	\eqref{eq:adjointlipschitz}, we apply \Cref{lem:lem41}
	with $y_2 = S(u_2,\xi)$ and $y_1 = S(u_1,\xi)$.
	Using the resulting estimate and \eqref{eq:Nov2520211145},
	and dividing \eqref{eq:adjointlipschitz} and
	\eqref{eq:Dec1520211631} by $\kappa_{\min}$, 
	 we obtain \eqref{eq:Nov2520211221}.

	Next, we show that $L_{\nabla \rpobj}$ defined in 
	\eqref{eq:nablarpobjLipschitz} is a Lipschitz constant
	of $\nabla_u \rpobj(\cdot,\xi)$ on $\adcsp$.
	\Cref{lem:Nov252021204} ensures
	\begin{align*}
		\nabla_u \rpobj(u_2,\xi)-\nabla_u\rpobj(u_1,\xi)
		= -g(\xi)\big(z(u_2,\xi)-z(u_1,\xi)\big).
	\end{align*}
	Combining \Cref{lem:Nov252021204,lem:s_unique_stable},
	\Cref{itm:adcspbounded},
	\friedrichs' inequality, and \eqref{eq:Nov2520211221},
	 we obtain the Lipschitz bound.
\end{proof}

\subsection{Existence of solutions}
\label{subsect:existence}
\Cref{lem:existencesolutions} establishes the existence of solutions
to the stochastic program \eqref{eq:ocp} 
and to the SAA problem
\eqref{eq:saa}.
\begin{proposition}
	\label{lem:existencesolutions}
	If \Cref{itm:psi,assumption:pde} hold, then
	 the stochastic program \eqref{eq:ocp} has a solution
	and for each $N \in \naturals$, the SAA problem
	\eqref{eq:saa} has a solution.
\end{proposition}
\begin{proof}
	Using \Cref{lem:properties_B}, \eqref{eq:Feb0420210958}, \eqref{eq:eq37}, 
	\eqref{eq:rpobj}, and \cite[Lem.\ 2.33]{Bonnans2013}, the assertions can be
	verified using  standard arguments. We omit the details.
\end{proof}

\section{Compact subset of feasible set}
\label{sec:compatset}
We construct a deterministic, compact set containing the
SAA critical points using first-order necessary optimality conditions
and the stability estimates established in \cref{sec:rnpdeopt}.
The compact subset is used to derive finite sample size estimates
in \cref{sec:samplesizeestimates}. While 
the computations performed in this section mainly
use the stability estimate \eqref{eq:Dec2820211103'},
the compact set's construction  provides an integral step towards 
establishing our  nonasymptotic sample size estimates
via the covering number approach. In particular, we derive
an explicit bound on the compact set's diameter.

We define the set of SAA critical points $\hat{\mathscr{C}}_N$ by
\begin{align*}
\hat{\mathscr{C}}_N
&= \{\, u_N \in \adcsp \colon \, 
u_N = \prox{\psi/\alpha}{-(1/\alpha)\nabla \hat{\erpobj}_N(u_N)}
\,\}.
\end{align*}
Let \Cref{ass:adcspbounded,assumption:pde} hold true.
Let  $u_N^* \in \adcsp$ be a local solution to
the SAA problem \eqref{eq:saa}.
Since  $\alpha > 0$ and  $\hat{\erpobj}_N$ is
\frechet\ differentiable according to \Cref{prop:efunsdifferentiable}, 
we have $u_N^* \in \hat{\mathscr{C}}_N$ (cf.\ \cite[p.\ 2092]{Mannel2020} and
\cite[Thm.\ 1.46]{Hinze2009}).
The set $\hat{\mathscr{C}}_N$ can be viewed as a set-valued mapping
from $\Omega$ to $L^2(\domain)$.
We also define the set $\hat{\mathscr{D}}_N$  by
\begin{align}
\label{eq:DN}
\hat{\mathscr{D}}_N
&=  \{\, v_N \in L^2(\domain) \colon \, 
v_N = -(1/\alpha)\nabla \hat{\erpobj}_N(\prox{\psi/\alpha}{v_N}) \,\}.
\end{align}
As with $\hat{\mathscr{C}}_N$, 
	the set $\hat{\mathscr{D}}_N$ can be viewed as a set-valued mapping from 
	$\Omega$ to $L^2(\domain)$.
We have the relationships (cf.\ \cite[p.\ 2092]{Mannel2020})
\begin{align}
\label{eq:normalcritical}
\hat{\mathscr{C}}_N =
\{\, \prox{\psi/\alpha}{v_N} 
\colon v_N \in \hat{\mathscr{D}}_N \, \}
= \prox{\psi/\alpha}{\hat{\mathscr{D}}_N}. 
\end{align}
We define the problem-dependent parameters
\begin{align}
\label{eq:parameters}
\begin{aligned}
\mathfrak{D}^{\mathscr{D}}
& = 
(C_\domain/\kappa_{\min}) \rrhs_{\max}
+ (C_\domain^2/\kappa_{\min}) g_{\max}  \rad
+ \norm[L^2(\domain)]{y_d},
\\
\radius^{\mathscr{D}}
& = 
(C_\domain+1)^2
(C_\domain/\kappa_{\min})g_{\max}
\mathfrak{D}^{\mathscr{D}}.
\end{aligned}
\end{align}
\Cref{prop:cinwadcsp} demonstrates that 
$\hat{\mathscr{D}}_N$ is contained in the set
\begin{align}
\label{eq:Feb0420211009}
\begin{aligned}
\vadcsp^{\mathscr{D}} = 
\{\,
& u \in H^1(\domain)  :\, 
\norm[H^1(\domain)]{u}
\leq 
(1/\alpha)\radius^{\mathscr{D}}
\,\}.
\end{aligned}
\end{align}
The set  $\vadcsp^{\mathscr{D}}$ is a compact subset of $L^2(\domain)$, 
as $\alpha > 0$, $\radius^{\mathscr{D}}$ is finite
and the embedding operator of the embedding 
$H^1(\domain) \embedding L^2(\domain)$ is compact.
\begin{proposition}
	\label{prop:cinwadcsp}
	If \Cref{ass:adcspbounded,assumption:pde} hold,
	then for each $N \in \naturals$,  it holds that $\hat{\mathscr{D}}_N\subset
	\vadcsp^{\mathscr{D}}$,
	where $\hat{\mathscr{D}}_N$ and $\vadcsp^{\mathscr{D}}$
	are defined in \eqref{eq:DN} and \eqref{eq:Feb0420211009}, respectively.
\end{proposition}%
\begin{proof}
	Fix $(u,\xi) \in \adcsp \times \Xi$.
	Using \eqref{eq:Dec2820211103'}, 
	$\norm[H^1(\domain)]{y} \leq 
	(C_\domain+1)\norm[H_0^1(\domain)]{y}$
	being valid for all $y \in H_0^1(\domain)$, and
	the definition of $\radius^{\mathscr{D}}$, we find that
	\begin{align}
	\label{eq:Dec1020211843}
	\norm[H^1(\domain)]{\nabla_u \rpobj(u,\xi)}
	&\leq \radius^{\mathscr{D}}.
	\end{align}	
	Fix $N \in \naturals$ and
	$v_N \in \hat{\mathscr{D}}_N$. 
	We have
	$v_N = -(1/\alpha)\nabla \hat{\erpobj}_N(\prox{\psi/\alpha}{v_N})$.
	Defining $u_N = \prox{\psi/\alpha}{v_N}$,
	we have $u_N \in \adcsp$. Choosing $u = u_N$
	in \eqref{eq:Dec1020211843}, we obtain
	$v_N \in \vadcsp^{\mathscr{D}}$.
\end{proof}

\Cref{prop:cinwadcsp}
and \eqref{eq:normalcritical} imply 
that for each $N \in \naturals$, 
the set $\hat{\mathscr{C}}_N$ is contained in 
$\prox{\psi/\alpha}{\vadcsp^{\mathscr{D}}}$.
This set is compact, as it is the image
of the compact set $\vadcsp^{\mathscr{D}}$ under 
$\prox{\psi/\alpha}{}$.

\section{Quantitative Sobolev embeddings}
\label{sec:covering}

The Sobolev 
embedding $H^1(\domain) \embedding L^2(\domain)$ is compact, provided
that $\domain \subset \real^d$ is a bounded Lipschitz domain. 
The authors of \cite{Birman1967,Birman1980}
establish covering numbers of closed 
unit balls in Sobolev spaces with respect
to Lebesgue norms.
\Cref{thm:birman1967}, which is an excerpt of
\cite[Thm.\ 1.7]{Birman1980}, is  the key result for establishing
the nonasymptotic sample size estimates 
in \cref{sec:samplesizeestimates}.

\begin{theorem}[{see \cite[Thm.\ 1.7]{Birman1980}}]
	\label{thm:birman1967}
	Let $s > 0$.
	The binary logarithm of the $\nu$-covering number
	of the closed $H^s((0,1)^d)$-unit ball
	with respect to the $L^2((0,1)^d)$-norm
	is proportional to  $(1/\nu)^{d/s}$
	for all (sufficiently small) $\nu > 0$.
\end{theorem}%

Note that in \cite[p.\ 2]{Birman1980} the definition
$W^{k,p}([0,1)^d) = W^{k,p}((0,1)^d)$ is made.
\Cref{thm:birman1967} implies the existence of a constant $\varrho > 0$
such that for all $\nu > 0$,
\begin{align}
\label{eq:eq5118}
\mathcal{N}(\nu; B_{H^1((0,1)^d)}, \norm[L^2((0,1)^d)]{\cdot})
\leq 2^{\varrho(1/\nu)^d},
\end{align}
where $B_{H^1((0,1)^d)}$ is the closed 
$H^1((0,1)^d)$-unit ball.
The upper bound in \Cref{thm:birman1967} is established in 
\cite[Thm.\ 5.2]{Birman1967}  for the  $H^s((0,1)^d)$-unit sphere
rather than for the closed $H^s((0,1)^d)$-unit ball.
We refer the reader to  \cite[sect.\ 1.3.12]{Novak1988}
and \cite[pp.\ 118 and 151]{Edmunds1996}
for related quantitative Sobolev embedding statements.

\section{Nonasymptotic sample size estimates}
\label{sec:samplesizeestimates}
We establish sample size estimates for SAA critical points.
We define the normal map $\phi : L^2(\domain) \to L^2(\domain)$ by
\begin{align}
\label{eq:phi}
\phi(v) = \nabla \erpobj(\prox{\psi/\alpha}{v}) + \alpha v.
\end{align}
If $v^* \in L^2(\domain)$ satisfies $\phi(v^*) = 0$, 
then $u^* = \prox{\psi/\alpha}{v^*}$
is a critical point of the stochastic program \eqref{eq:ocp}
\cite[p.\ 2092]{Mannel2020}.

\begin{theorem}
	\label{thm:samplesize}
	Let \Cref{ass:adcspbounded,assumption:pde} hold and let $\domain = (0,1)^d$.
	Let $\varepsilon > 0$.
	If $\bar v_N$ is a measurable selection of 
	$\hat{\mathscr{D}}_N$
	defined in \eqref{eq:DN} and
	\begin{align}
	\label{eq:expectationboundN}
	N \geq 
	\frac{12\ln(2)\tau_{\mathscr{D}}^2}{\varepsilon^2}
	\bigg[\varrho
	\Big(\frac{4\max\{L_{\nabla \rpobj},1\}\radius^{\mathscr{D}}}
	{\alpha\varepsilon}\Big)^{d} + 1\bigg],
	\end{align}
	then
	\begin{align}
	\label{eq:expectationbound}
	\cE{\norm[L^2(\domain)]{\phi(\bar{v}_N)}}
	\leq \varepsilon,
	\end{align}
	where
	$\mathfrak{D}^{\mathscr{D}}$ and $\radius^{\mathscr{D}}$ are
	defined in \eqref{eq:parameters},
	$\varrho > 0$ is specified in \eqref{eq:eq5118},
	the Lipschitz constant $L_{\nabla \rpobj}$ is given in 
	\eqref{eq:nablarpobjLipschitz}, and
	\begin{align*}
	\tau_{\mathscr{D}} & = 2 
	(C_\domain^2/\kappa_{\min})g_{\max}
	\mathfrak{D}^{\mathscr{D}}.
	\end{align*}
	If  $\delta \in (0,1)$ and
	\begin{align}
	\label{eq:tailboundN}
	N \geq 
	\frac{48\tau_{\mathscr{D}}^2}{\varepsilon^2}
	\bigg[\ln(2)\varrho\Big(\frac{4L_{\nabla \rpobj}\radius^{\mathscr{D}}}
	{\alpha\varepsilon}\Big)^{d}
	+ \ln\Big(\frac{2}{\delta}\Big)\bigg],
	\end{align}
	then 
	\begin{align}
	\label{eq:tailbound}
	\Prob{\hat{\mathscr{D}}_N 
		\subset \mathscr{D}^\varepsilon}
	\geq 1-\delta,
	\end{align}
	where
	$\hat{\mathscr{D}}_N $ is defined in \eqref{eq:DN},
	$\phi$  in \eqref{eq:phi}, and
	$\mathscr{D}^\varepsilon
	=  \{\, v \in L^2(\domain) \colon \, 
	\norm[L^2(\domain)]{\phi(v)} \leq \varepsilon \,\}
	$.
\end{theorem}

Before establishing \Cref{thm:samplesize}, we briefly comment on
the sample size estimates \eqref{eq:expectationboundN}
and \eqref{eq:tailboundN}, and address measurability
issues in \Cref{lem:missues}.
The first addends  in the 
sample size estimates \eqref{eq:expectationboundN}
and \eqref{eq:tailboundN} are very similar.
Our proofs  of \Cref{thm:samplesize,prop:uniformboundsoperator} show that 
the Lipschitz constant $L_{\nabla \rpobj}$
(see \Cref{prop:efunsdifferentiable'}) in \eqref{eq:expectationboundN} 
may be replaced by the expected value of an integrable
Lipschitz constant of $\nabla_u \rpobj(\cdot,\xi)$ on $\adcsp$.
While this approach may result in a potentially 
less conservative sample size estimate
than \eqref{eq:expectationboundN}, we choose to use 
$L_{\nabla \rpobj}$ in \eqref{eq:expectationboundN}.

\begin{lemma}
	\label{lem:missues}
	If \Cref{ass:adcspbounded,assumption:pde} hold, then
	the set-valued mapping
	$\hat{\mathscr{D}}_N$ is measurable
	with closed, nonempty images and 
	$\{\, \omega \in \Omega \colon \, 
	\hat{\mathscr{D}}_N(\omega) \subset \mathscr{D}^\varepsilon\,
	\}$
	is measurable.
\end{lemma}
\begin{proof}
	Since $\nabla \hat{\erpobj}_N$ is a \Caratheodory\ function
	(see \Cref{lem:Feb0720211501,lem:Nov252021204}),
	$\adcsp$ is nonempty and closed, 
	and $\prox{\psi/\alpha}{}$ is firmly nonexpansive
	\cite[Prop.\ 12.28]{Bauschke2011},
	the set-valued mapping $\hat{\mathscr{D}}_N$ is
	closed-valued and measurable
	\cite[Thm.\ 8.2.9]{Aubin2009}.
	\Cref{prop:efunsdifferentiable} ensures
	that $\nabla_u \rpobj(\cdot, \xi)$ is Lipschitz continuous
	with a Lipschitz constant
	independent of $\xi \in \Xi$. 
	Hence, $\nabla \erpobj$ is (Lipschitz) continuous, 
	yielding the closeness of
	$\mathscr{D}^\varepsilon$.
	\Cref{lem:existencesolutions} implies that 
	$\mathscr{D}^\varepsilon$ and
	$\hat{\mathscr{D}}_N $ are nonempty.
	Now, \Cref{lem:inclusionmeasurable}
	ensures the measurability of the event
	$\{\, \hat{\mathscr{D}}_N \subset \mathscr{D}^\varepsilon \, \}$.
\end{proof}

\begin{proof}[{Proof of \Cref{thm:samplesize}}]
	\Cref{lem:missues} ensures that 
	$\{\, \hat{\mathscr{D}}_N \subset \mathscr{D}^\varepsilon\,
	\}$
	is measurable.
	To establish the expectation bound \eqref{eq:expectationbound}
	and the tail bound \eqref{eq:tailbound}, we apply
	\Cref{prop:uniformboundsoperator}
	with $\vadcsp = \vadcsp^{\mathscr{D}}$, 
	$\hsp = L^2(\domain)$,
	and $G = \nabla_u \rpobj$,
	where $\vadcsp^{\mathscr{D}}$ is defined in \eqref{eq:Feb0420211009}.	
	We verify the hypotheses of 
	\Cref{prop:uniformboundsoperator}.
	\Cref{assumption:basicerrorestimate1} is implied by
	\Cref{lem:Nov252021204,lem:Feb0720211501}.
	Using \eqref{eq:nablarpobjLipschitz}
	and \Cref{prop:efunsdifferentiable}, we find that
	$\nabla_u \rpobj(\cdot,\xi)$ is Lipschitz continuous
	with Lipschitz constant $L_{\nabla \rpobj}$ for all $\xi \in \Xi$. 
	Hence, 
	\Cref{assumption:basicerrorestimate2,assumption:basicerrorestimate6} hold
	true. 
	We verify \Cref{assumption:basicerrorestimate3}. 	
	By construction, the set $\vadcsp^{\mathscr{D}}$  
	is an  $H^1(\domain)$-ball about zero with radius
	$(1/\alpha)\radius^{\mathscr{D}}$.
	Therefore, \Cref{thm:birman1967} yields
	\begin{align}
	\label{eq:coveringvadcsp}
	\mathcal{N}(\nu; \vadcsp^{\mathscr{D}}, \norm[L^2(\domain)]{\cdot})
	\leq 2^{\varrho(\radius^{\mathscr{D}}/(\alpha\nu))^d}
	\quad \text{for all} \quad \nu > 0,
	\end{align}
	where $\varrho > 0$ is specified in \eqref{eq:eq5118}. 
	Hence, \Cref{assumption:basicerrorestimate3} holds true.
	We verify \Cref{assumption:basicerrorestimate4}.
	Fix $(u,\xi) \in \adcsp \times \Xi$.
	Using \eqref{eq:Feb0420210958} and
	\eqref{eq:Nov252021210}, we obtain
	$\norm[L^2(\domain)]{\nabla_u \rpobj(u,\xi)}\leq \tau_{\mathscr{D}}/2$.
	Thus,
	$\norm[L^2(\domain)]{\nabla_u \rpobj(u,\xi)-\cE{\nabla_u \rpobj(u,\xi)}}
	\leq \tau_{\mathscr{D}}$.
	Combined with 
	part \ref{itm:boundedsubgaussian_essentiallybounded} in
	\Cref{lem:boundedsubgaussian}, we find that 
	\Cref{assumption:basicerrorestimate4} holds true with 
	$\tau_G = \tau_{\mathscr{D}}$. 
	Finally,  \Cref{prop:cinwadcsp} ensures
	that for each $N \in \naturals$, $\hat{\mathscr{D}}_N \subset 
	\vadcsp^{\mathscr{D}}$. 
	Hence, we can use the estimate
	established in \Cref{rem:basicerrorestimate5}.
	
	To derive the expectation bound \eqref{eq:expectationbound},
	we use \eqref{eq:coveringvadcsp}
	and \eqref{eq:expsup}.
	We choose 
	$\nu = \varepsilon/(4\max\{L_{\nabla \rpobj},1\})$, 
	yielding $2L_{\nabla \rpobj} \nu \leq \varepsilon/2$.
	Using \eqref{eq:expsup}, \eqref{eq:coveringvadcsp}, and 
	the estimate established in \Cref{rem:basicerrorestimate5}, we
	have
	\begin{align}
		\label{eq:expecationnormalmap}
		\cE{\norm[L^2(\domain)]{\phi(\bar{v}_N)}}
		\leq \tfrac{\varepsilon}{2}
		+
		\tfrac{\sqrt{3}\tau_{\mathscr{D}}}{ \sqrt{N}}
		\sqrt{\ln(2)\big(\varrho(4\max\{L_{\nabla \rpobj},1\}
			\radius^{\mathscr{D}}/(\alpha\varepsilon))^d+1\big)}.
	\end{align}
	Requiring  the right-hand side be less than or equal to 
	$\varepsilon$ and solving
	for $N$ yields   \eqref{eq:expectationbound}.

	It remains to show \eqref{eq:tailbound}.
	Using \eqref{eq:coveringvadcsp}, the tail bound \eqref{eq:probsup'},
	and \Cref{rem:basicerrorestimate5}, we have:
	if $\bar{v}_N \in \hat{\mathscr{D}}_N$, then
	$\phi(\bar{v}_N) < \varepsilon$
	and hence $\bar{v}_N \in \mathscr{D}^\varepsilon$
	with a probability of at least
	$$1-2 \cdot 2^{\varrho(4L_{\nabla \rpobj}
		\radius^{\mathscr{D}}/(\alpha\varepsilon))^d}
	\eu^{-N\varepsilon^2/(48\tau_{\mathscr{D}}^2)}.
	$$
	Requiring this term be greater or equal to $1-\delta$
	and solving for $N$,  we obtain \eqref{eq:tailbound}.
\end{proof}

	\Cref{thm:samplesize} establishes sample size estimates 
	in terms of the normal map \eqref{eq:phi}. 
	Next we derive sample size estimates for another popular 
	criticality measure, which 
	can be used within termination criteria for numerical methods,
	for example. In \cref{sect:numillus}, we provide numerical illustrations
	using this criticality measure.
We define the criticality measure $\chi : L^2(\domain) \to \real$
for \eqref{eq:ocp} by
\begin{align}
	\label{eq:cmeasure}
	\chi(u) = 
	\norm[L^2(\domain)]{u- \prox{\psi/\alpha}{-(1/\alpha)\nabla \erpobj(u)}}
\end{align}
and  the set of $\varepsilon$-critical points of
\eqref{eq:ocp} by
$\mathscr{C}^\varepsilon = \{\, u \in \adcsp \colon \, 
\chi(u) \leq \varepsilon
\, \}$ for $\varepsilon \geq 0$.
\begin{corollary}
	\label{cor:samplesize}
	Under the hypotheses of \Cref{thm:samplesize},
	and $\varepsilon > 0$, it follows that
	\emph{(a)}
	if 
	$\bar u_N$ is a measurable selection of $\hat{\mathscr{C}}_N $
	and $N$ satisfies \eqref{eq:expectationboundN},
	then
	$
	\cE{\chi(\bar u_N)}
	\leq \varepsilon/\alpha
	$,
	and
	\emph{(b)}
	if  $\delta \in (0,1)$ and
	$N$ fulfills \eqref{eq:tailboundN}, 
	then 
	$
	\Prob{\hat{\mathscr{C}}_N 
		\subset \mathscr{C}^{(\varepsilon/\alpha)}}
	\geq 1-\delta
	$.
\end{corollary}
\begin{proof}
	The measurability of $\{\, \hat{\mathscr{C}}_N 
	\subset \mathscr{C}^{(\varepsilon/\alpha)} \,\}$
	can be established using arguments  similar  to those in the proof of
	\Cref{lem:missues}. Fix $v \in L^2(\domain)$ and define $u = 
	\prox{\psi/\alpha}{v}$.
	Since the proximity operator is firmly nonexpansive 
	\cite[Prop.\ 12.28]{Bauschke2011}, we have
	\begin{align}
	\label{eq:chiphi}
	\begin{aligned}
	\chi(u)
	& = 
	\norm[L^2(\domain)]{\prox{\psi/\alpha}{v}-\prox{\psi/\alpha}
		{-(1/\alpha)\nabla \erpobj(u)}}
	\\
	& \leq \norm[L^2(\domain)]{v+(1/\alpha)\nabla \erpobj(u)}
	 = (1/\alpha)\norm[L^2(\domain)]{\phi(v)},
	\end{aligned}
	\end{align}
	where $\phi$ is defined in \eqref{eq:phi}.
	Since $\bar{u}_N$ is a measurable selection of  $\hat{\mathscr{C}}_N$,
	\Cref{lem:missues}
	and Filippov's theorem \cite[Thm.\ 8.2.10]{Aubin2009}
	applied to the identity in \eqref{eq:normalcritical}
	ensure the existence 
	of a  measurable selection $\bar{v}_N$ of $\hat{\mathscr{D}}_N$
	defined in \eqref{eq:DN} such that
	$\bar{u}_N = \prox{\psi/\alpha}{\bar{v}_N}$.
	Now \Cref{thm:samplesize} implies the assertions.
\end{proof}

\section{Numerical illustrations}
\label{sect:numillus}
The main purposes of our numerical illustrations are to verify the typical Monte
Carlo convergence rate for $\cE{\chi(\bar u_N)}$ 
where $\bar u_N$ is an SAA critical point
and to  examine the dependence of our expectation 
bounds on the  parameter $\alpha$.
For numerical computations, 
the infinite dimensionality of \eqref{eq:ocp} and its SAA problem
\eqref{eq:saa} necessitate finite dimensional approximations. 
We also illustrate empirically that the expectation bounds 
are independent of
the dimension of the finite dimensional spaces.
	The estimate \eqref{eq:expecationnormalmap} 
	when combined with \eqref{eq:chiphi} yields
	for each $\varepsilon > 0$,
	\begin{align}
	\label{eq:expectationchi}
	\cE{\norm[L^2(\domain)]{\chi(\bar{u}_N)}}
	\leq \tfrac{\varepsilon}{2\alpha}
	+
	\tfrac{\sqrt{3}\tau_{\mathscr{D}}}{\alpha \sqrt{N}}
	\sqrt{\ln(2)\big(\varrho(4\max\{L_{\nabla \rpobj},1\}
		\radius^{\mathscr{D}}/(\alpha\varepsilon))^d+1\big)},
	\end{align}
	where $\chi$ is defined in \eqref{eq:cmeasure}. For fixed
	$\alpha > 0$ and $\varepsilon > 0$, 
	the second term in the right-hand side in \eqref{eq:expectationchi} 
	decays with the rate 
	$1/\sqrt{N}$ as a function of $N$. Moreover, for fixed
	$N \in \naturals$ and $\varepsilon > 0$, 
	the first term increases
	with rate $1/\alpha$
	and the second term with rate $1/\alpha^{d/2+1}$ as $\alpha$ approaches
	zero.

\begin{figure}[t]
	\centering
	\subfloat{%
		\includegraphics[width=0.465\textwidth]
		{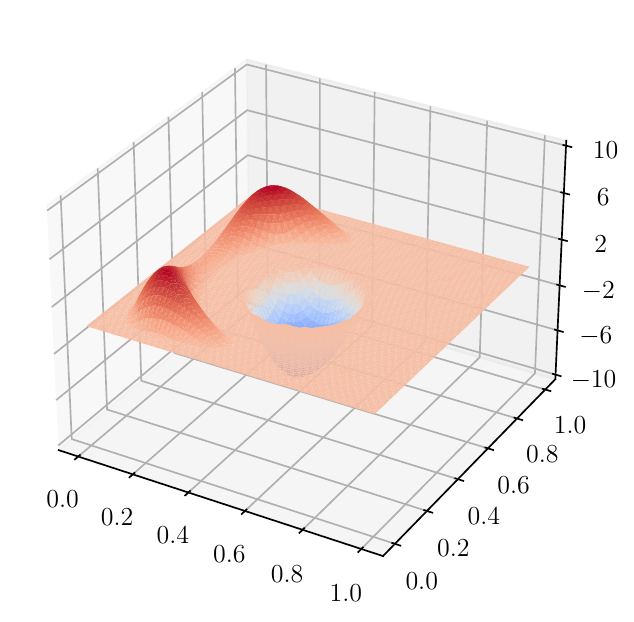}}
	\subfloat{%
		\includegraphics[width=0.465\textwidth]
		{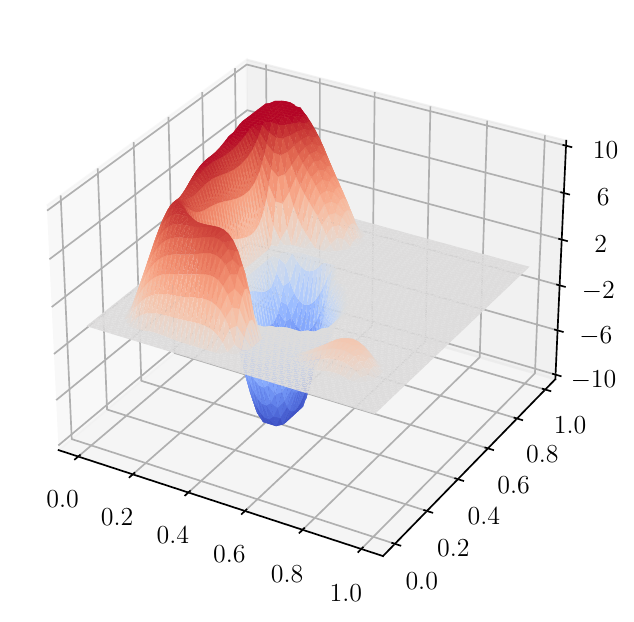}}
	\caption{For 
		$\alpha = 10^{-3}$
		and discretization parameter $n=64$,
		nominal critical point, that is,
		a critical point of \eqref{eq:nom} 
		\textnormal{(left)}, and a reference critical point
		of \eqref{eq:ocp}, that is, 
		a critical point 
		of \eqref{eq:saasob} \textnormal{(right)}.}
	\label{fig:nomref}
\end{figure}

We consider an instance of \eqref{eq:ocp}. 
Let $d = 2$ and $\domain = (0,1)^d$. For $\gamma \geq 0$, 
let $\psi(u) = \gamma \norm[L^1(\domain)]{u}$
if $u \in L^2(\domain)$ with $-10 \leq u \leq 10$
and $\psi(u) =  \infty$ otherwise. 
We define
$\Xi = [-1,1]^{100}$,  $\nop(t) = t^3$,
and
$\dualpHzeroone{A(\xi)y}{v} 
 = \int_\domain \kappa(x,\xi) \nabla y(x)^T \nabla v(x) \du x
$,
where $\kappa : \bar{\domain} \times \Xi \to \real$ is given by
$\kappa(x,\xi) = 
	\eu^{
		\sum_{k=1}^{25} 5/(2k^2) \sin(4k\xi_k \pi x_1)\sin(4k\pi \xi_{25+k}x_2)
	}
$
if $x_1 \leq 1/2$
and $\kappa(x,\xi) = 
3/2 +  \sum_{k=1}^{25}|(10/k^2)\xi_{25+k} \cos((10+\xi_{25+k})x_1x_2)|
$ otherwise.
We further define  $y_d(x) = -1$ if $x \in [1/4, 3/4]^2$
and $y_d(x) = 1$ otherwise,  
$b(x,\xi) = 1+\sum_{k=1}^{25}(5/k^2)\xi_{75+k} x_1x_2\cos(4\pi k x_1)\sin(4\pi k x_2)$
if $x_1 \leq 3/4 +\xi_{76}/2$
and
$b(x,\xi) = 1+|
\sum_{k=1}^{25} (3/k^2)x_2\xi_{75+k} \sin(3\pi x_2)\cos(3\pi(x_1-k\xi_{75+k}x_2))|
$
otherwise, 
and 
$g(x,\xi) = \max\{1,
\sum_{k=1}^{25} (10/k^2) \xi_{50+k} \sin((4+k)\xi_{50+k}x_1x_2)
\cos((4+k)\xi_{50+k}x_1x_2)
\}$.
The random variables $\xi_1, \ldots, \xi_{100}$ are independent
$[-1,1]$-uniformly distributed. 
The random fields $\kappa$, $b$, and $g$ are nonsmooth. Their 
definitions have been guided by the desire to design an instance
of \eqref{eq:Feb0320211603} with random fields that
 lack a representation as a  linear combination
of a moderate number of $L^2(\domain)$-orthogonal basis functions
and of  functions that are separable  with respect to the
vector $x $ and the parameters $\xi_1, \ldots, \xi_{100}$. 
In particular, each of the random fields $\kappa$, $b$, and $g$
lacks a respresentation as a truncated Karhunen--Lo\`eve-type expansion
with a small number of addends.

We discretized \eqref{eq:ocp} and \eqref{eq:saa} using a finite element
discretization;
$H_0^1(\domain)$ is approximated using piecewise linear continuous
finite elements and $L^2(\domain)$ is discretized using piecewise constant
functions defined  on a regular triangular mesh with  
a total of $2n^2$ triangles. 
Here, $n$ is the number of triangles in each direction in $(0,1)^2$.
For a fixed $n \in \mathbb{N}$, 
the finite dimensional control space is denote by $\csp_n$.
We use the subscript $n$ in $\rpobj_n$
and $\nabla \hat{\erpobj}_{N,n}$ to denote the approximations
of  $\rpobj$ (see \eqref{eq:rpobj}) and 
$\nabla \hat{\erpobj}_N$, respectively, resulting from the finite dimensional
approximations.

\begin{figure}[t]
	\centering
	\subfloat{%
		\includegraphics[width=0.465\textwidth]
		{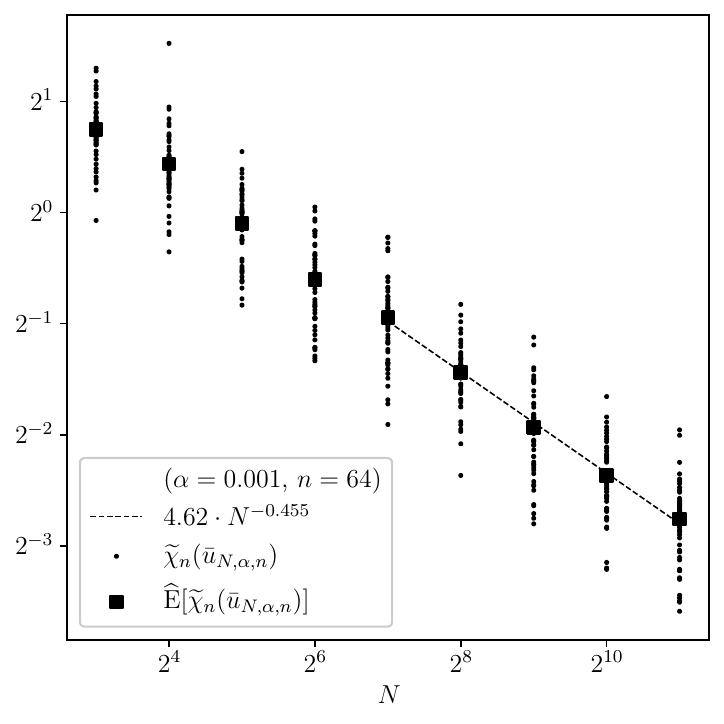}}
	\subfloat{%
		\includegraphics[width=0.465\textwidth]
		{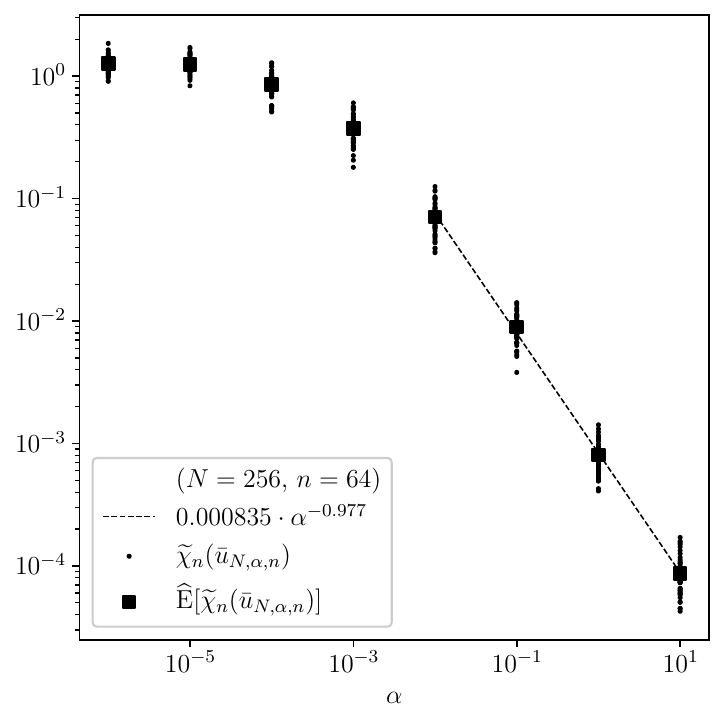}}
	\caption{
		For $n = 64$, empirical estimate
		$\widehat{\mathrm{E}}[\widetilde{\chi}_n(\bar u_{N,\alpha,n})]$
		of $\cE{\widetilde{\chi}_n(\bar u_{N,\alpha,n})}$
		over $N$ with $\alpha = 10^{-3}$
		\textnormal{(left)}
		and empirical estimate 
		$\widehat{\mathrm{E}}[\widetilde{\chi}_n(\bar u_{N,\alpha,n})]$
		of $\cE{\widetilde{\chi}_n(\bar u_{N,\alpha,n})}$
		over $\alpha$ with $N = 256$
		\textnormal{(right)}.
		Here, $\bar u_{N,\alpha,n}$ are SAA critical points
			of \eqref{eq:saafinitedim}
			computed with sample size $N$, regularization parameter
			$\alpha$, and discretization
			parameter $n$. The approximated criticality measure
			$\widetilde{\chi}_n$ is defined in \eqref{eq:chiapprox}. For both plots, $48$ independent realizations of
		the criticality measure $\widetilde{\chi}_n(\bar u_{N,\alpha,n})$ are depicted.
		The convergence rates were computed using least squares.
		For each plot, the first four empirical means were excluded 
		for the least squares computations.
	}
	\label{fig:errors}
\end{figure}

\begin{figure}[t]
	\centering
	\subfloat{%
		\includegraphics[width=0.465\textwidth]
		{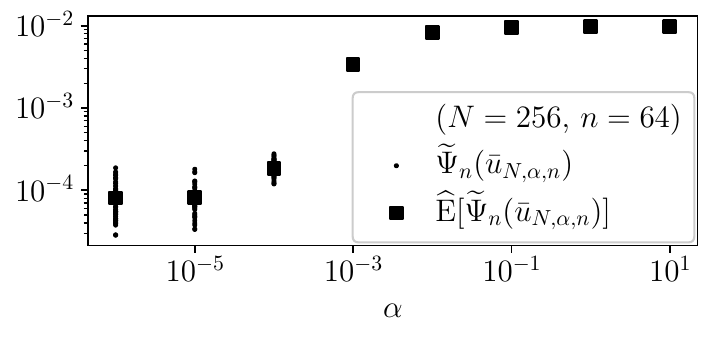}}
	\subfloat{%
		\includegraphics[width=0.465\textwidth]
		{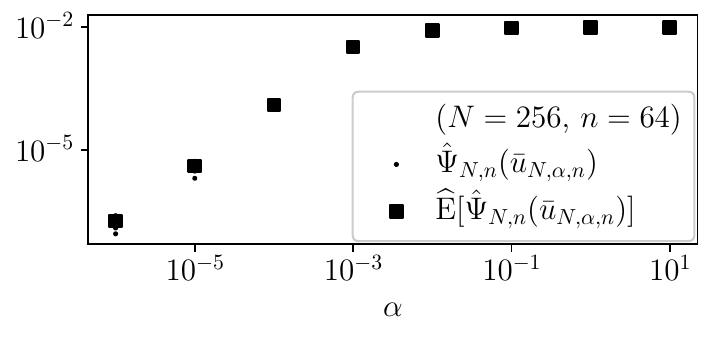}}
	\caption{%
		For $n = 64$ and $N=256$, empirical mean
		$\widehat{\mathrm{E}}[\widetilde{\Psi}_{n}(\bar u_{N,\alpha,n})]$
		of
		$\mathbb{E}[\widetilde{\Psi}_{n}(\bar u_{N,\alpha,n})]$
		over $\alpha$
		\textnormal{(left)} and
		empirical mean
		$\widehat{\mathrm{E}}[\hat{\Psi}_{N,n}(\bar u_{N,\alpha,n})]$
		of
		$\mathbb{E}[\hat{\Psi}_{N,n}(\bar u_{N,\alpha,n})]$
		over $\alpha$
		\textnormal{(right)}.
		Here, $\bar u_{N,\alpha,n}$ are the SAA critical points
		of \eqref{eq:saafinitedim} used in
		\Cref{fig:errors} \textnormal{(right)}.
		The  criticality measures
		$\widetilde{\Psi}_{n}$ 
		and
		$\hat{\Psi}_{N,n}$
		are defined in \eqref{eq:tikhonovapprox}.
	}
	\label{fig:tikhonov}
\end{figure}

The criticality measure $\chi$ in  \eqref{eq:cmeasure}
involves the gradient $\nabla \erpobj(\cdot) = \cE{\nabla_u \rpobj(\cdot,\xi)}$.
To approximate $\chi$, we generated $N_1 = 2^{13}$ samples
using a Sobol' sequence \cite{Joe2008} and transformed
the samples to take values in $\Xi$.
Let us denote these samples by
 $\widetilde{\xi}_1, \ldots, \widetilde{\xi}_{N_1}$.
Defining $\widetilde{\erpobj}_{N_1,n}(\cdot) = (1/N_1) \sum_{i=1}^{N_1}
\rpobj_n(\cdot,\widetilde{\xi}_i)
$, the criticality measure $\chi$ is approximated by
$\widetilde{\chi}_n : L^2(\domain) \to [0,\infty)$ defined by
\begin{align}
	\label{eq:chiapprox}
	\widetilde{\chi}_{n}(u) = 
	\norm[L^2(\domain)]{u- \prox{\psi/\alpha}{-(1/\alpha)
			\nabla \widetilde{\erpobj}_{N_1,n}(u)}}.
\end{align}
We also use the samples  $\widetilde{\xi}_1, \ldots, \widetilde{\xi}_{N_1}$ 
to approximate the risk-neutral PDE-constrained
optimization problem \eqref{eq:ocp} by the finite dimensional SAA problem
\begin{align}
	\label{eq:saasob}
	\min_{u \in \csp_n}\, 
	\frac{1}{N_1} \sum_{i=1}^{N_1}
	\rpobj_n(u,\widetilde{\xi}_i)
	+ (\alpha/2) \norm[L^2(\domain)]{u}^2
	+ \psi(u).
\end{align}
Finally, we approximate the 
infinite dimensional SAA problems \eqref{eq:saa} by
the finite dimensional SAA problems
\begin{align}
	\label{eq:saafinitedim}
	\min_{u \in \csp_n}\, 
	\frac{1}{N} \sum_{i=1}^N\rpobj_n(u,\xi^i)
	+ (\alpha/2) \norm[L^2(\domain)]{u}^2
	+ \psi(u).
\end{align}

Inspired by the choice made in 
\cite[p.\ 199]{Parikh2014}, 
we chose $\gamma = 0.2\hat{\gamma}_{\max}$,
where 
$\hat{\gamma}_{\max} = 
\norm[L^\infty(\domain)]{\nabla \hat{\erpobj}_{10,64}(0)}$.
We rounded $0.2\hat{\gamma}_{\max}$ to three significant figures, yielding
$\gamma = 7.48\cdot 10^{-3}$. This value is used for all simulations.
The computations used to generate the simulation output
depicted in 
\Cref{fig:nomref,fig:errors,fig:dimension2,fig:dimension}
were performed on a
Linux cluster with 48 CPUs (Intel Xeon CPU E7-8857 v2 3.00GHz) and 1TB of RAM\@.
\Cref{fig:tikhonov} is based on output of simulations
performed on the  PACE Phoenix cluster \cite{PACE2017}.
We used \href{http://www.dolfin-adjoint.org/}
{\texttt{dolfin-adjoint}} 
\cite{Mitusch2019,Funke2013}
with \href{https://fenicsproject.org/}{\texttt{FEniCs}}
\cite{Alnes2015,Logg2012} 
to evaluate the cost functions
and their derivatives. The problems \eqref{eq:saasob} and \eqref{eq:saafinitedim}
were solved using a
semismooth Newton-CG method applied to a normal map
\cite{Mannel2020,Ulbrich2011}. Its implementation  is based on
that of \href{https://github.com/funsim/moola}{\texttt{Moola}}'s 
\texttt{NewtonCG}\ method \cite{Nordaas2016}.\footnote{%
	Our computer code,
	simulation output, and figures depicting samples
	of $\kappa$, $g$, and $b$ are available at \url{https://github.com/milzj/SAA4PDE/tree/semilinear_complexity/simulations/semilinear_complexity}.} 
We present numerical illustrations for moderate values 
of the space discretization parameter $n$ 
and sample size $N$. These parameter choices are aimed at balancing
the empirical demonstration of our theoretical
findings with computational resource use.

\begin{figure}[t]
	\centering
	\subfloat{%
		\includegraphics[width=0.465\textwidth]{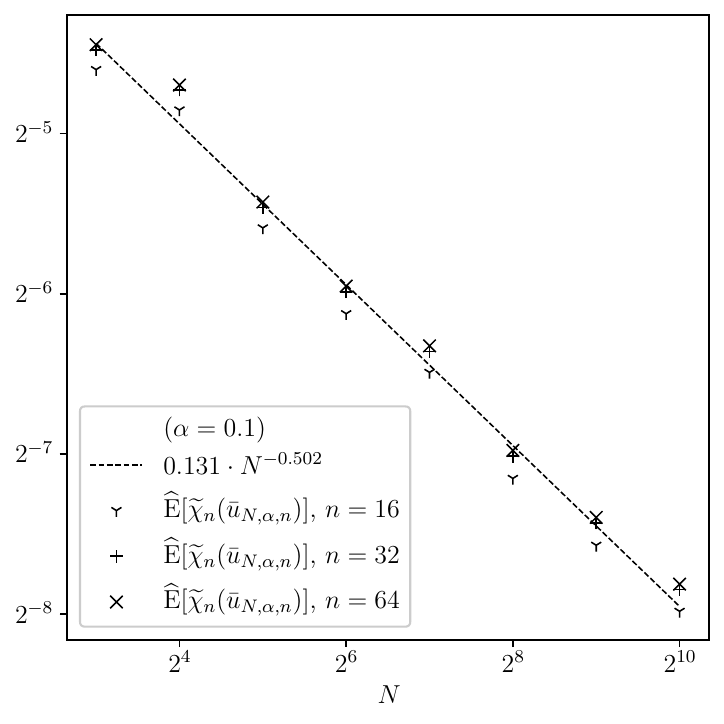}}
	\subfloat{%
		\includegraphics[width=0.465\textwidth]
		{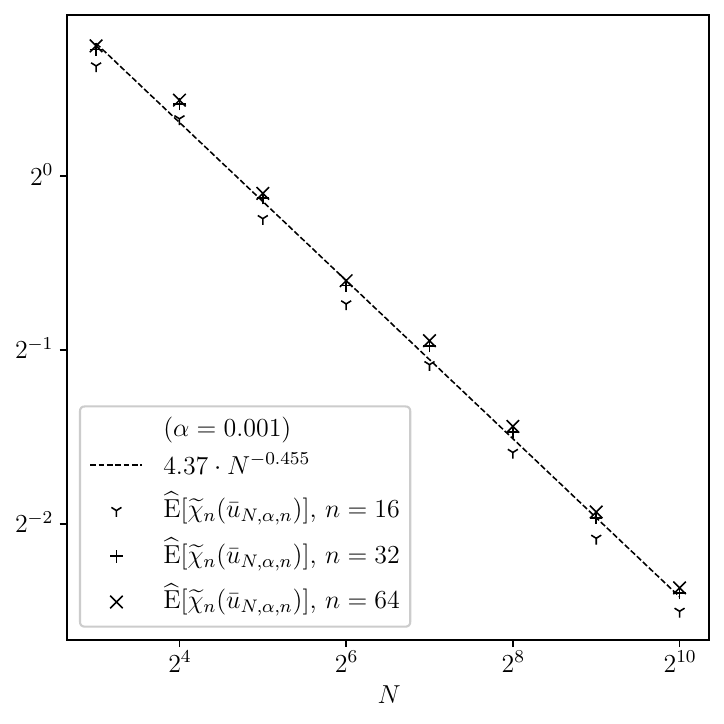}}
	\caption{
			For multiple values of the discretization parameter $n$,
			empirical estimates
			$\widehat{\mathrm{E}}[\widetilde{\chi}_n(\bar u_{N,\alpha,n})]$
			of $\cE{\widetilde{\chi}_n(\bar u_{N,\alpha,n})}$
			over $N$ with $\alpha = 10^{-1}$
			\textnormal{(left)}
			and with $\alpha = 10^{-3}$
			\textnormal{(right)}.
			Here, $\bar u_{N,\alpha,n}$ are SAA critical points
			of \eqref{eq:saafinitedim}
			computed with sample size $N$, regularization parameter
			$\alpha$, and discretization
			parameter $n$. The approximated criticality measure
			$\widetilde{\chi}_n$ is defined in \eqref{eq:chiapprox}.
			After averaging the empirical estimates
			for each  $N$, 
			the convergence rates were computed using least squares.
	}
	\label{fig:dimension2}
\end{figure}

For visualization, critical points were interpolated to
the discretized state space. For $\alpha = 10^{-3}$
and $n=64$,
\Cref{fig:nomref} depicts a nominal critical point, that is, 
a critical point of the
nominal problem
\begin{align}
	\label{eq:nom}
	\min_{u \in \csp_n}\,
	\rpobj_n(u,\cE{\xi})+(\alpha/2)\norm[L^2(\domain)]{u}^2 + \psi(u)
\end{align}
and a critical point of \eqref{eq:saasob}.
Throughout the section, 
we estimate  $\cE{\widetilde{\chi}_n(\bar u_{N,\alpha,n})}$
using $48$ independent realizations of
the SAA critical point $\bar u_{N,\alpha,n}$
of \eqref{eq:saafinitedim} and denote the
empirical estimate by 
$\widehat{\mathrm{E}}[\widetilde{\chi}_n(\bar u_{N,\alpha,n})]$.

For $n=64$,
\Cref{fig:errors} depicts the empirical estimate of
$\cE{\widetilde{\chi}_n(\bar u_{N,\alpha,n})}$ with $\alpha = 10^{-3}$
and shows the typical Monte Carlo
convergence rate $1/\sqrt{N}$.
\Cref{fig:errors} also depicts
the empirical estimate of $\cE{\widetilde{\chi}_n(\bar u_{N,\alpha,n})}$
for various values of $\alpha$, which highlights a dependence
on the expected error of SAA critical points on the 
parameter $\alpha$. 
To shed some light on the convergence behavior of the empirical estimate
of $\cE{\widetilde{\chi}_n(\bar u_{256,\alpha,64})}$
for $\alpha \to 0^+$ as depicted in \Cref{fig:errors}, 
we  interpret \eqref{eq:ocp}, \eqref{eq:saasob}, and \eqref{eq:saafinitedim}
as PDE-constrained optimization problems
resulting from a Tikhonov regularization \cite[pp.\ 29--37]{Dontchev1993}
with regularization parameter $\alpha$ and 
empirically demonstrate that $\bar u_{256,\alpha,64}$ provide
approximate critical points of the optimization problems
\begin{align}
	\label{eq:saafinitedimzero}
	\min_{u \in \csp_n}\, 
	\frac{1}{N_1} \sum_{i=1}^{N_1}
	\rpobj_n(u,\widetilde{\xi}_i)
	+ \psi(u)
	\quad \text{and} \quad 
	\min_{u \in \csp_n}\, 
	\frac{1}{N} \sum_{i=1}^N\rpobj_n(u,\xi^i)
	+ \psi(u)
\end{align}
as $\alpha \to 0^+$. We define 
$\widetilde{\Psi}_{n}$, $\hat{\Psi}_{N,n} \colon L^2(\domain) \to [0,\infty)$ by
\begin{align}
	\label{eq:tikhonovapprox}
	\begin{aligned}
	\widetilde{\Psi}_{n}(u) & = 
	\norm[L^2(\domain)]{u- \prox{\psi}{u-
			\nabla \widetilde{\erpobj}_{N_1,n}(u)}}, 
	\\
	\hat{\Psi}_{N,n}(u) & = 
	\norm[L^2(\domain)]{u- \prox{\psi}{u-
			\nabla \hat{\erpobj}_{N,n}(u)}}.
	\end{aligned}
\end{align}
The function $\widetilde{\Psi}_{n}$ is 
a criticality measure for the first problem in \eqref{eq:saafinitedimzero}
and $\hat{\Psi}_{N,n}$ is one
for the second problem in \eqref{eq:saafinitedimzero};
cf., e.g., \cite[sect.\ 2.1.2]{Milzarek2019}.
\Cref{fig:tikhonov} depicts the 
empirical mean of 
$\widetilde{\Psi}_{64}(\bar u_{256,\alpha,64})$
and $\hat{\Psi}_{256,64}(\bar u_{256,\alpha,64})$
computed using $48$ independent SAA critical points $\bar u_{256,\alpha,64}$
for various values of $\alpha$. This may suggest
that the controls $\bar u_{256,\alpha,64}$ provide
approximate critical points of the optimization
problems in \eqref{eq:saafinitedimzero} as $\alpha \to 0^+$.

\Cref{fig:dimension2} depicts the empirical estimates
of $\cE{\widetilde{\chi}_n(\bar u_{N,\alpha,n})}$
over the sample size 
$N$ for multiple values of the discretization parameter $n$.
For fixed sample sizes $N$ and regularization parameters $\alpha$, 
\Cref{fig:dimension} depicts
empirical estimates of $\cE{\widetilde{\chi}_n(\bar u_{N,\alpha,n})}$ for multiple discretization parameters $n$. 
The simulation output in \Cref{fig:dimension2,fig:dimension} 
indicates that the expectation
$\cE{\widetilde{\chi}_n(\bar u_{N,\alpha,n})}$ may be independent
of the discretization parameter $n$ and hence of the dimension $2n^2$ of
the finite dimensional control space $\csp_n$.
The theoretical analysis of this empirical observation
is beyond the scope of this manuscript 
and is left as a topic for future research.
For risk-neutral linear PDE-constrained
optimization, the SAA approach
with finite dimensional approximations
of the control and state spaces are analyzed in 
\cite{Hoffhues2020,Martin2021,Milz2022c} with respect to
proximity of finite dimensional SAA solutions to the ``true'' optimal
control.

\begin{figure}[t]
	\centering
	\subfloat{%
		\includegraphics[width=0.465\textwidth]
		{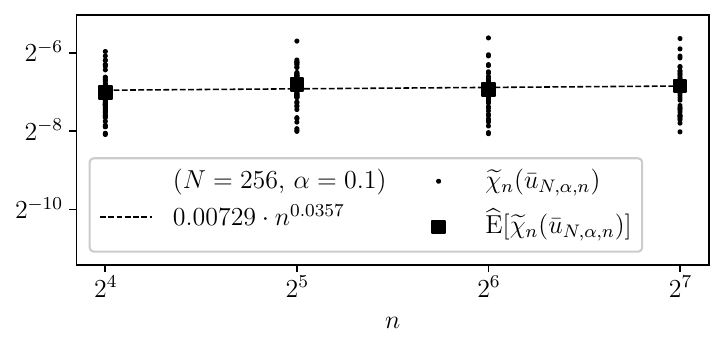}}
	\subfloat{%
		\includegraphics[width=0.465\textwidth]
		{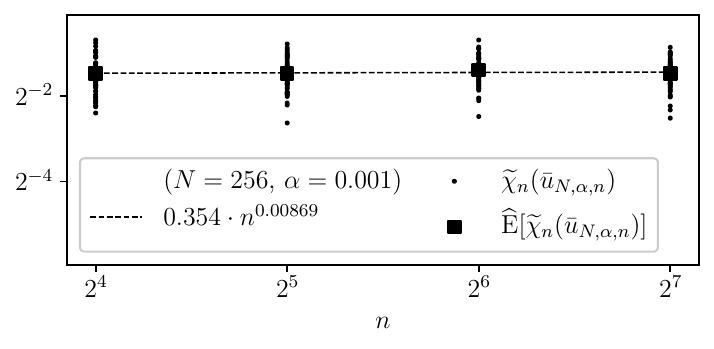}}
	\caption{
		Empirical estimate
		$\widehat{\mathrm{E}}[\widetilde{\chi}_n(\bar u_{N,\alpha,n})]$
		of $\cE{\widetilde{\chi}_n(\bar u_{N,\alpha, n})}$
		over the  discretization parameter $n$ with
		$N = 256$ and $\alpha = 10^{-1}$
		\textnormal{(left)}
		and with $N = 256$ and $\alpha = 10^{-3}$
		\textnormal{(right)}.
		Here, $\bar u_{N,\alpha,n}$ are SAA critical points
			of \eqref{eq:saafinitedim}
			computed with sample size $N$, regularization parameter
			$\alpha$, and discretization
			parameter $n$. The approximated criticality measure
			$\widetilde{\chi}_n$ is defined in \eqref{eq:chiapprox}.
		For both plots, $48$ independent realizations of
		the criticality measure $\widetilde{\chi}_n(\bar u_{N,\alpha,n})$ are depicted.
		The convergence rates were computed using  least squares.
	}
	\label{fig:dimension}
\end{figure}

\section{Discussion}
\label{sec:discussion}
We established nonasymptotic sample size estimates 
for  SAA critical points of
risk-neutral semilinear PDE-constrained optimization, a
class of infinite dimensional, nonconvex stochastic programs. 
To derive the sample size estimates,
we constructed a compact subset of the feasible
set containing all SAA critical points. 
Using covering numbers of Sobolev function classes, we
derived nonasymptotic sample size estimates inspired
by the analyses in \cite{Shapiro2003,Shapiro2005,Shapiro2021}. 
The construction of the compact
set exploits structure in PDE-constrained optimization problems.
This structure has been used for other purposes
in different ways, such as for establishing
mesh-independence of semismooth Newton methods \cite{Hintermueller2004},
developing smoothing steps within semismooth
Newton methods \cite{Ulbrich2011}, and
deriving higher regularity
of solutions to deterministic PDE-constrained optimization problems
\cite{Casas2012a}.

	Risk-neutral semilinear PDE-constrained optimization 
	is a subfield of infinite dimensional 
	optimization with dynamical systems under uncertainty.
	Our approach to establishing sample size estimates 
	may be applied to risk-neutral PDE-constrained optimization problems
	other than those considered here. Among other things, our derivations
	rely on the Lipschitz continuity properties of the reduced parameterized
	objective function's gradient and covering numbers of
	a set containing all SAA critical points. 
	These key properties may be verified
	for objective functions other than tracking-type functions
	\eqref{eq:rpobj} and parameterized operator equations other than 
	the semilinear PDE \eqref{eq:Feb0320211603}. 
	
	Risk-averse PDE-constrained optimization 
	\cite{Garreis2021,Kouri2018,Kouri2016} provides a more general approach
	to optimization of complex systems under uncertainty
	than risk-neutral optimization. 
	For risk-averse semilinear 
	PDE-constrained optimization using the average/conditional 
	value-at-risk, our
	analysis when combined with considerations
	in \cite[sect.\ 4.2]{Lan2012}
	may be adapted to constructing a compact set
	containing all SAA critical points. 
	While deriving  sample size estimates for critical points is 
	complicated by the risk-averse objective function's nonsmoothness, 
	establishing sample size estimates for optimal values
	and optimal solutions may be possible using uniform exponential
	tail bounds.

Using \Cref{thm:birman1967}, 
the dependence on the dimension $d$ in the sample size estimates
\eqref{eq:expectationboundN} and \eqref{eq:tailboundN} 
could be improved to $d/2$ if 
there exists $\radius > 0$
such that $\nabla_u \rpobj(u,\xi) \in H^2(\domain)$
and $\norm[H^2(\domain)]{\nabla_u \rpobj(u,\xi)}
	\leq \radius$
for all $(u,\xi) \in \adcsp \times \Xi$. However, the constant 
$\radius^{\mathscr{D}}$ in
\eqref{eq:expectationboundN} and in \eqref{eq:tailboundN} 
must then be replaced by $\radius$ which could be significantly larger
than $\radius^{\mathscr{D}}$. For a class of linear elliptic PDEs
with random inputs,
$H^2(\domain)$-stability estimates are provided in
\cite{Teckentrup2013}. These stability estimates could be 
used to establish $\norm[H^2(\domain)]{\nabla_u \rpobj(u,\xi)}
\leq \radius$
for all $(u,\xi) \in \adcsp \times \Xi$, provided that
the parameterized elliptic operator 
$A$ satisfies additional assumptions.

\appendix

\section{Measurability of set inclusions}
\label{sec:setinclusion}
We establish a measurability statement of set inclusions which 
is used in \cref{sec:samplesizeestimates}.
This statement is essentially known.
As introduced in \cref{sec:notation},
$(\Theta, \mathcal{A}, \mu)$ is a complete 
probability space.

\begin{lemma}
	\label{lem:inclusionmeasurable}
	Let $\dsp$ be a real, separable  Banach space and
	let $\Gamma : \Theta \rightrightarrows \dsp$
	be a
	measurable set-valued mapping with closed, nonempty images.
	If  $\Upsilon \subset \dsp$ is nonempty and closed, then
	$\{\, \theta \in \Theta\colon \, \Gamma(\theta) \subset \Upsilon \, \}$
	is measurable. 
\end{lemma}
\begin{proof}
	Since $\dsp \setminus \Upsilon$ is open
	and $\Gamma$ is  measurable, 
	$\Gamma^{-1}(\dsp \setminus \Upsilon)$ is measurable.
	Combined with
	$
	\{\, \theta \in \Theta\colon \, \Gamma(\theta) \subset \Upsilon \, \}
	= 
	\Theta \setminus \Gamma^{-1}(\dsp \setminus \Upsilon)
	$, we find that
	$\{\, \theta \in \Theta\colon \, \Gamma(\theta) \subset \Upsilon \, \}$
	is measurable.
\end{proof}

\section{Sub-Gaussian-type  bounds for maxima of random vectors}
\label{sec:subgaussianbounds}
We establish expectation and exponential tail bounds
for pointwise maxima of sub-Gaussian-type Hilbert space-valued
random vectors. 
The techniques used to derive
these results are similar to those used
to establish expectation and tail bounds
for pointwise maxima of sub-Gaussian real-valued random variables.
However, we use \cite[Thm.\ 3]{Pinelis1986} which provides
expectation bounds for sums of independent, mean-zero
Hilbert space-valued random vectors.
The results established in this section are
used in \cref{sect:uniformtailbounds} to derive uniform 
expectation and exponential tail bounds.

\begin{proposition}
	\label{prop:maxmean}
	Let $\tau > 0$, let $\hsp$ be a real, separable Hilbert space and
	for each $k   \in \{1,2, \ldots, K\}$,
	let
	$\rvv_{i,k} : \Theta\to \hsp$
	($i = 1, 2, \ldots, N$)
	be independent, mean-zero random 
	vectors with
	$\cE{\cosh(\lambda \norm[\hsp]{\rvv_{i,k}})}\leq \exp(\lambda^2\tau^2/2)$ 
	for all $\lambda \in \real$ 
	and $i \in \{  1, 2, \ldots, N \}$. 
	We define $\rv_k^{[N]} = (1/N) \sum_{i=1}^N \rvv_{i,k}$.
	Then for each $\varepsilon \geq 0$,
	\begin{align}
	\label{eq:maxmean}
	\cE{\max_{1\leq k \leq K} \norm[\hsp]{\rv_k^{[N]} }}
	& \leq \sqrt{3/2}\tau (1/\sqrt{N})\sqrt{2\ln(2K)},
	\\
	\label{eq:maxprob}
	\Prob{\max_{1\leq k \leq K} \norm[\hsp]{\rv_k^{[N]}} \geq \varepsilon}
	& \leq 2K \exp(-\tau^{-2}\varepsilon^2N/3).
	\end{align}
\end{proposition}

The proof of \Cref{prop:maxmean} is presented at the end of the section. 
Let $\rv : \Theta \to \real$  be a random variable.
We provide a motivation for the condition
\begin{align}
\label{eq:subgaussiantype}
\cE{\cosh(\lambda |\rv|)}\leq \exp(\lambda^2\tau^2/2)
\quad \tfa \quad \lambda \in \real,
\end{align}
where $\tau >  0$.
A random variable $\rv : \Theta \to \real$
is sub-Gaussian with parameter $\tau$
if $\tau \in [0, \infty)$
and  $\cE{\exp(\lambda \rv)} \leq \exp(\lambda^2\tau^2/2)$
for all $\lambda \in \real$ \cite[p.\ 9]{Buldygin2000}.
Since 
$\cosh(x) = (\eu^{-x}+\eu^x)/2$ for all $x \in \real$, 
a sub-Gaussian random variable $\rv$ with parameter $\tau > 0$
fulfills \eqref{eq:subgaussiantype}. 
If $\hsp$ is a real, separable Hilbert space
and $\rvv : \Theta \to \hsp $  is a nondegenerate, 
mean-zero Gaussian random vector, 
then \eqref{eq:subgaussiantype} holds
with $\rv = \norm[\hsp]{\rvv}$
and $\tau = \cE{\norm[\hsp]{\rvv}^2}^{1/2}$
\cite[Rem.\ 4]{Pinelis1986}.
Suppose that there exists $\sigma > 0$ such that
\begin{align}
\label{eq:subgaussian}
\cE{\exp( |\rv|^2/\sigma^2)} \leq \eu.
\end{align}
This condition and its variants are used 
in the literature on stochastic programming
\cite{Nemirovski2009,Shapiro2021,Lan2020}
with $\rv$ being the norm of a stochastic (sub)gradient,
for example.
Let us demonstrate that \eqref{eq:subgaussiantype}
and \eqref{eq:subgaussian} are essentially equivalent.
\begin{lemma}
	\label{lem:boundedsubgaussian}
	Let $\rv : \Theta \to \real$  be a random variable.
	\begin{enumerate}[nosep,leftmargin=*]
		\item If \eqref{eq:subgaussiantype} holds
		with $\tau > 0$, then \eqref{eq:subgaussian}
		holds with
		$\sigma = (2\tau^2/(1-\eu^{-2}))^{1/2}$.
		\item If \eqref{eq:subgaussian} is satisfied with
		$\sigma > 0$, then 
		\eqref{eq:subgaussiantype} 
		is fulfilled with $\tau = 2^{1/4}\sigma$.
		\item 
		\label{itm:boundedsubgaussian_essentiallybounded}
		If $\tau > 0$  and  $|\rv| \leq \tau$ \wpone, then
		\eqref{eq:subgaussiantype} is fulfilled.
	\end{enumerate}
\end{lemma}
\begin{proof}
	\begin{enumerate}[nosep,leftmargin=*,wide]
		\item 		
		We adapt the proof of \cite[Lem.\ 1.6 on p.\ 9]{Buldygin2000}.
		We define $ s= 1-\eu^{-2}$
		and $\sigma^2 = 2\tau^2/s$.
		We have for all $z \in \real$
		(cf.\  \cite[p.\ 9]{Buldygin2000}),
		\begin{align*}
			\int_{\real}\eu^{\lambda^2\tau^2(s-1)/(2s)}\du \lambda
			= \tfrac{1}{\tau} \sqrt{\tfrac{2\pi s}{1-s}}
			\quad \tand \quad 
			\int_\real \cosh(\lambda z)\eu^{-\lambda^2\tau^2/(2s)} \du \lambda 
			=  \tfrac{\sqrt{2\pi s}}{\tau}\eu^{z^2/\sigma^2}.
		\end{align*}
		Multiplying \eqref{eq:subgaussiantype} by 
		$\eu^{-\lambda^2\tau^2/(2s)}$ yields 
		$\cE{\cosh(\lambda|\rv|)\eu^{-\lambda^2\tau^2/(2s)}} 
		\leq \eu^{\lambda^2\tau^2(s-1)/(2s)}$
		for all $\lambda \in \real$.
		Integrating both sides over $\lambda \in \real$
		and using Fubini's theorem, we obtain
		$
		\tfrac{\sqrt{2\pi s}}{\tau} \cE{\eu^{|\rv|^2/\sigma^2}}
		\leq 
		\tfrac{1}{\tau} \sqrt{\tfrac{2\pi s}{1-s}}
		$.
		Hence,
		$\cE{\eu^{|\rv|^2/\sigma^2}} \leq 1/(1-s)^{1/2} = \eu$.
				\item 
		The proof is inspired by that of
		\cite[Prop.\ 9.81]{Shapiro2021}.
		Fix $\lambda \in \real$
		with $\lambda^2\sigma^2/2 \leq 1$.
		Using Jensen's inequality and
		\href{https://tinyurl.com/4xxc2t4e}{$\cosh(x) \leq \exp(x^2/2)$
			for all $x \in \real$}
		\cite[eq.\ (4.6.6)]{Tropp2015}, 
		\begin{align*}
			\cE{\cosh(\lambda|\rv|)} \leq 
			\cE{\eu^{\lambda^2|\rv|^2/2}}
			= \cE{\eu^{\lambda^2\sigma^2|\rv|^2/(2\sigma^2)}}
			\leq
 \cE{\eu^{|\rv|^2/\sigma^2}}^{\lambda^2\sigma^2/2}
 		\leq \eu^{\lambda^2\sigma^2/2}.
		\end{align*}
		Now let $\lambda \in \real$
		with $\lambda^2\sigma^2/2 > 1$.
		Young's inequality ensures
		$2\lambda s \leq \lambda^2\sigma^2/\sqrt{2} + \sqrt{2}s^2/\sigma^2$
		 for all $s \in \real$. Combined with
		$\cosh(x) \leq \exp(x)$ being valid for all $x \geq 0$,
		the symmetry of $\cosh$, 
		$\sqrt{2}\lambda^2\sigma^2/4  >  \sqrt{2}/2$,
		 and Jensen's inequality, we have
		\begin{align*}
			\cE{\cosh(\lambda|\rv|)}
			\leq \eu^{\lambda^2\sigma^2/(2\sqrt{2})}
			\cE{\eu^{ \sqrt{2}|\rv|^2/(2\sigma^2)}}
			\leq  \eu^{\lambda^2\sigma^2/(2\sqrt{2}) + \sqrt{2}/2}
			< \eu^{\sqrt{2}\lambda^2\sigma^2/2}.
		\end{align*}
		Putting together the pieces, we obtain the assertion.
		
		\item Since $\cosh(x) \leq \exp(x^2/2)$
		for all $x \in \real$ \cite[eq.\ (4.6.6)]{Tropp2015}
		and $\cosh$
		is a symmetric function, we have
		$\cE{\cosh(\lambda|\rv|)} \leq 
		\cosh(\lambda\tau) \leq \exp(\lambda^2\tau^2/2)$
		for all $\lambda \in \real$.
	\end{enumerate}
\end{proof}

\begin{lemma}
	\label{prop:saa:2020-11-21T20:25:01.71}
	Let $\tau > 0$, 
	let $\hsp$ be a real, separable Hilbert space, and let
	$\rv_i : \Theta\to \hsp$
	be independent, mean-zero random vectors such that
	$\cE{\cosh(\lambda \norm[\hsp]{\rv_i})}\leq \exp(\lambda^2\tau^2/2)$ 
	for all $\lambda \in \real$   ($i = 1, 2, \ldots, N$). Then, for each
	 $\lambda \in \real$,
	$$
	\cE{\cosh(\lambda\norm[\hsp]{ \rv_1 + \cdots + \rv_N})} 
	\leq \exp(3\lambda^2 \tau^2 N/4).
	$$
\end{lemma}

\Cref{prop:saa:2020-11-21T20:25:01.71} is 
established using \Cref{thm:saa:2020-03-15T00:12:53.514}.

\begin{theorem}
	[{see \cite[Thm.\ 3]{Pinelis1986}}]
	\label{thm:saa:2020-03-15T00:12:53.514}
	Let $\hsp$ be a real, separable Hilbert space.
	If $\rv_i : \Theta \to \hsp$ 
	$(i=1, \ldots, N)$ are independent, mean-zero 
	random vectors, then
	$$
	\cE{\cosh(\lambda\norm[\hsp]{ \rv_1 + \cdots + \rv_N})} 
	\leq \prod_{i=1}^N 
	\cE{\exp(\lambda \norm[\hsp]{\rv_i})- \lambda \norm[\hsp]{\rv_i}}
	\quad \tfa \quad \lambda \geq 0.
	$$
\end{theorem}

\begin{proof}[{Proof of \Cref{prop:saa:2020-11-21T20:25:01.71}}]
	Fix $\lambda \geq 0$. We have \href{https://tinyurl.com/y4xq8tne}
	{$\exp(s)-s \leq \cosh(\sqrt{3/2} s )$
	for all $s \in \real$}.
	Hence,
	$
	\cE{\exp(\lambda \norm[\hsp]{\rv_i}) -\lambda \norm[\hsp]{\rv_i}}
	\leq \exp(3\lambda^2\tau^2/4)
	$.
	Combined with \Cref{thm:saa:2020-03-15T00:12:53.514}, 
	we find that 
	$
	\cE{\cosh(\lambda\norm[\hsp]{ \rv_1 + \cdots + \rv_N})} 
	\leq 
	\prod_{i = 1}^N  \exp(3\lambda^2 \tau^2/4) 
	= \exp(3\lambda^2 \tau^2 N/4)
	$.
	Since $\cosh$ is symmetric, the estimate is valid for all
	$\lambda \in \real$.
\end{proof}

	\Cref{lem:meanmax} establishes an expectation and an exponential
	tail bound for pointwise maxima of
	sub-Gaussian random variables similar to those found, e.g., in
	\cite[Prop.\ 7.29]{Foucart2013} and 
	\cite[sect.\ 2.5]{Boucheron2013}, 
	but we express 	sub-Gaussianity via \eqref{eq:subgaussiantype}. 

\begin{lemma}
	\label{lem:meanmax}
	Let $\sigma > 0$.
	Suppose that $\rv_k : \Theta \to \real$
	are random variables
	with $\cE{\cosh(\lambda |Z_k|)} \leq \exp(\lambda^2\sigma^2/2)$
	for all $\lambda \in \real$ and
	$k = 1, 2, \ldots, K$.
	Then for all $\varepsilon > 0$,
	\begin{align*}
	\cE{\max_{1\leq k \leq K} |Z_k|}
	\leq \sigma \sqrt{2\ln(2K)}
	\quad \tand \quad 
	\Prob{\max_{1\leq k \leq K} |Z_k| \geq \varepsilon}
	\leq 2K \eu^{-\varepsilon^2/(2\sigma^2)}.
	\end{align*}
\end{lemma}
\begin{proof}
	The proof of the expectation bound uses standard derivations
	based on ``smoothing''
	of the pointwise
	maximum.
	We have $\exp \leq 2\cosh$.
	For $\lambda > 0$, we have
	\begin{align*}
	\cE{\max_{1\leq k \leq K} |Z_k|}
	& \leq
	(1/\lambda) \ln \Big(\sum_{k=1}^K \cE{\exp(\lambda|Z_k|)} \Big)
	\leq 
	(1/\lambda) \ln \Big(\sum_{k=1}^K 2\cE{\cosh(\lambda|Z_k|)} \Big)
	\\
	& \leq 
	(1/\lambda) \ln \big(2K \exp(\lambda^2\sigma^2/2) \big)
	=
	(1/\lambda) \ln(2K) 
	+  \lambda^2\sigma^2/2.
	\end{align*}
	Choosing $\lambda = \sqrt{2\ln(2K)}/\sigma$ yields the
	expectation bound.

	The union bound
	and Markov's inequality ensure for all $\lambda > 0$,
	\begin{align*}
	\Prob{\max_{1\leq k \leq K} |Z_k| \geq \varepsilon}
	&\leq \sum_{k=1}^K 
	\Prob{|Z_k| \geq \varepsilon}
	\leq \eu^{-\lambda\varepsilon} 
	\sum_{k=1}^K 
	\cE{\exp(\lambda |Z_k|)}
	\\
	& \!\! \leq  2\eu^{-\lambda\varepsilon} 
	\sum_{k=1}^K 
	\cE{\cosh(\lambda |Z_k|)}
	\leq  2K \eu^{-\lambda\varepsilon} \eu^{\lambda^2\sigma^2/2}
	= 2K \eu^{-\lambda\varepsilon + \lambda^2\sigma^2/2}.
	\end{align*}
	Minimizing the right-hand side over $\lambda > 0$ yields
	the tail bound.
\end{proof}

\begin{proof}[{Proof of \Cref{prop:maxmean}}]
	We prove \eqref{eq:maxmean} and \eqref{eq:maxprob} using
	\Cref{prop:saa:2020-11-21T20:25:01.71,lem:meanmax}.
	Fix $\lambda \geq 0$.
	\Cref{prop:saa:2020-11-21T20:25:01.71}
	and 
	$\cE{\cosh(\lambda \norm[\hsp]{\rvv_{i,k}})}\leq \exp(\lambda^2\tau^2/2)$ 
	imply
	\begin{align*}
	\cE{\cosh(\lambda\norm[\hsp]{\rv_k^{[N]}})}
	\leq \exp(3\lambda^2 \tau^2 /(4N)).
	\end{align*}
	Defining $\sigma = \tau \sqrt{3/2}/\sqrt{N}$, 
	we have $\exp(3\lambda^2 \tau^2 /(4N)) 
	= \exp(\lambda^2\sigma^2/2)$.
	Now we apply \Cref{lem:meanmax}
	to the random variables $\norm[\hsp]{\rv_k^{[N]}}$, 
	yielding \eqref{eq:maxmean} and \eqref{eq:maxprob}.
\end{proof}

\section{Uniform exponential tail bounds for SAA normal maps}
\label{sect:uniformtailbounds}
We derive uniform expectation and exponential tail bounds for
an SAA normal map in Hilbert spaces.
The derivation of the tail bounds is inspired by those 
in \cite[Thms.\ 9.84 and 9.86]{Shapiro2021}
for real-valued functions. The uniform exponential tail bounds are used
in \cref{sec:samplesizeestimates} to derive nonasymptotic sample size
estimates.

Throughout the section, $\xi$, $\Xi$, $\xi^1, \xi^2, \ldots$
and $(\Omega, \cF, P)$ are as in 
\cref{sec:intro,sec:rnpdeopt}.
Let $\hsp$ be a real, separable
Hilbert space and let $\varphi: \hsp \to (-\infty,\infty]$
be proper, convex and lower semicontinuous. 
We define $\adhsp = \{\, u \in \hsp \colon \, \varphi(u) < \infty \, \}$.
Let  $G : \adhsp \times \Xi \to \hsp$ be a function. 
We define
$\hat{G}_N : \adhsp \to \hsp$ by
$\hat{G}_N(u) = (1/N)\sum_{i=1}^N G(u,\xi^i)$.
Let $\alpha > 0$.
We further define $\Phi$, $\hat{\Phi}_N : \hsp \to \hsp$ by
\begin{align*}
\Phi(v)
= \cE{G(\prox{\varphi/\alpha}{v},\xi)}
\quad \text{and} \quad 
\hat{\Phi}_N(v)
= \hat{G}_N(\prox{\varphi/\alpha}{v}). 
\end{align*}
Since $\xi^1, \xi^2, \ldots$ are defined on $\Omega$, we can
view $\hat{\Phi}_N$ as a function on $\hsp \times \Omega$.
The mappings $v \mapsto \Phi(v) + \alpha v$
and $v \mapsto \hat{\Phi}_N(v) + \alpha v$ define normal maps
\cite{Robinson1992}.
When establishing
the uniform tail bounds
in \Cref{prop:uniformboundsoperator}, we only rely on properties of their
difference, that is, on characteristics of  $\Phi-\hat{\Phi}_N$.
The following assumptions are inspired by those used
in \cite[sect.\ 9.2.11]{Shapiro2021}.
A random variable $\rv : \Theta \to \real$
is sub-exponential with parameters $(\tau, \Lambda)$
if $\tau \in [0, \infty)$, $\Lambda \in (0, \infty]$
and $\cE{\exp(\lambda Z)} \leq \exp(\lambda^2\tau^2/2)$
for all $\lambda \in (-\Lambda, \Lambda)$ \cite[p.\ 19]{Buldygin2000}.

\begin{assumption}[{Uniform exponential tail bounds: Problem data}]
	\label{assumption:basicerrorestimate}
	\begin{enumthm}[nosep,leftmargin=*]
		\item 
		\label{assumption:basicerrorestimate1}
		The mapping $G : \adhsp \times \Xi \to \hsp$
		is a \Caratheodory\ function, 
		and $G(u,\xi)$ is Bochner integrable for each $u \in \adhsp$.
		
		\item 
		\label{assumption:basicerrorestimate2}
		For an integrable random variable $M : \Xi \to (0,\infty)$, 
		\begin{align*}
		\norm[\hsp]{G(u_2,\xi)-G(u_1,\xi)} 
		\leq  M(\xi) \norm[\hsp]{u_2-u_1}
		\quad \tfa \quad u_2, u_1 \in \adhsp, \;\; \xi \in \Xi. 
		\end{align*}
		\item 
		\label{assumption:basicerrorestimate6}
		The random variable $M(\xi)-\cE{M(\xi)}$
		is sub-exponential with parameters
		$(\tau_M, \Lambda_M)$.
		\item 
		\label{assumption:basicerrorestimate4}
		There exists $\tau_G > 0$ such that,
		for each $u \in \adhsp$, 
		\begin{align}
		\label{eq:cEcosh}
		\cE{\cosh(\lambda \norm[\hsp]{G(u,\xi)-\cE{G(u,\xi)}})}
		\leq \exp(\lambda^2\tau_G^2/2)
		\quad \tfa \quad \lambda \in \real.
		\end{align}
		\item 
		\label{assumption:basicerrorestimate3}
		The $\nu$-covering number 
		of the nonempty, closed set $\vadcsp \subset \adhsp$ 
		is finite for all $\nu > 0$.
	\end{enumthm}
\end{assumption}

\Cref{assumption:basicerrorestimate3} ensures that
$\vadcsp$ is compact, as it is closed and totally bounded 
\cite[Lem.\ 8.2-2]{Kreyszig1978}. 	
We define $\hat{M}_N = (1/N) \sum_{i=1}^N M(\xi^i)$.
\begin{lemma}
	\label{lem:expecationgcontinuous}
	If \Cref{assumption:basicerrorestimate1,assumption:basicerrorestimate2}
	hold, then $\Phi$
	and $\hat{\Phi}_N$
	are Lipschitz continuous  with Lipschitz constants
	$\cE{M(\xi)}$ and $\hat{M}_N$, respectively.
	Moreover 
	$\sup_{v \in  \vadcsp}\, \norm[\hsp]{\hat{\Phi}_N(v)-\Phi(v)}$
	is measurable
	on $\Omega$.
\end{lemma}
\begin{proof}
	For each $u \in \adhsp$, $\cE{G(u,\xi)}$ is well-defined. 
	Hence, 
	\Cref{assumption:basicerrorestimate2} ensures 
	that $\cE{G(\cdot,\xi)}$
	is Lipschitz continuous on $\adhsp$ with Lipschitz constant
	$\cE{M(\xi)}$.
	Since $\prox{\varphi/\alpha}{}$
	is firmly nonexpansive
	and $\prox{\varphi/\alpha}{\hsp} \subset \adhsp$,
	$\Phi$ is Lipschitz continuous with Lipschitz constant
	$\cE{M(\xi)}$. Similarly, we obtain that
	$\hat{\Phi}_N$ is Lipschitz continuous
	with Lipschitz constant $\hat{M}_N$.
	Combined with the fact that 
	$G$ is a \Caratheodory\ mapping
	on $\adhsp \times \Xi$ and that
	$\xi^i: \Omega \to \Xi$ ($i=1, 2, \ldots$) are random elements,
	we find that 
	$\norm[\hsp]{\hat{\Phi}_N(\cdot)-\Phi(\cdot)}$
	is a \Caratheodory\ map
	on $\adhsp \times \Omega$.
	Since  $\vadcsp \subset \adhsp$ 
	is nonempty and closed, 
	$\sup_{v \in  \vadcsp}\, \norm[\hsp]{\hat{\Phi}_N(v)-\Phi(v)}$
	is measurable \cite[Lem.\ 2.1]{Hiai1977}.
\end{proof}
\begin{proposition}
	\label{prop:uniformboundsoperator}
	Let \Cref{assumption:basicerrorestimate} hold.
	Then $\Phi$ is Lipschitz continuous
	with Lipschitz constant $L = \cE{M(\xi)}$ and
	for all $\nu > 0$,
	\begin{align}
	\label{eq:expsup}
	\cE{\sup_{v \in \vadcsp}\, \norm[\hsp]{\hat{\Phi}_N(v)-\Phi(v)}}
	\leq 2L\nu+  
	\tfrac{\sqrt{3}\tau_G}{\sqrt{N}}
	\sqrt{\ln(2\mathcal{N}(\nu; \vadcsp, \norm[\hsp]{\cdot}))}.
	\end{align}	
	Define  $\ell = L^2/(2\tau_M^2)$
	if $L \leq \Lambda_M \tau_M^2$
	and $\ell = \Lambda_M L/2$  otherwise.
	Then for all $\varepsilon > 0$,
	\begin{align*}
	\Prob{\sup_{v \in \vadcsp}\, \norm[\hsp]{\hat{\Phi}_N(v)-\Phi(v)}
		\geq \varepsilon}
	\leq 
	\eu^{-N\ell}
	+ 2\mathcal{N}(\varepsilon/(4L); \vadcsp, \norm[\hsp]{\cdot})
	\eu^{-N\varepsilon^2 /(48\tau_G^2)}.
	\end{align*}
	If furthermore $M(\xi) \leq L$ for all $\xi \in \Xi$, then
	for all $\varepsilon > 0$,
	\begin{align}
	\label{eq:probsup'}
	\Prob{\sup_{v \in \vadcsp}\, \norm[\hsp]{\hat{\Phi}_N(v)-\Phi(v)}
		\geq \varepsilon}
	\leq 
	2\mathcal{N}(\varepsilon/(4L); \vadcsp,\norm[\hsp]{\cdot})
	\eu^{-N\varepsilon^2 /(48\tau_G^2)}.
	\end{align}
\end{proposition}

Before establishing
\Cref{prop:uniformboundsoperator}, we illustrate how
it is used in \cref{sec:samplesizeestimates} to derive nonasymptotic sample size
estimates.
\begin{remark}
	\label{rem:basicerrorestimate5}
	Let the hypotheses of \Cref{prop:uniformboundsoperator} hold.
	Suppose that \wpone\ for each $N \in \naturals$, 
	$\{\, v \in \hsp \colon 
	\, \hat{\Phi}_N(v) + \alpha v= 0 \, \}$ 
	is contained in $\vadcsp$.
	Let $N \in \naturals$ and let $v_N \in \hsp$ be a measurable selection
	of 	$\{\, v \in \hsp \colon 
	\, \hat{\Phi}_N(v) + \alpha v= 0 \, \}$.
	Then \wpone,
	\begin{align*}
	\norm[\hsp]{\alpha v_N + \Phi(v_N)}
	&\leq \norm[\hsp]{\hat{\Phi}_N(v_N)-\Phi(v_N)} + 
	\norm[\hsp]{\hat{\Phi}_N(v_N)+\alpha v_N}
	\\
	&\leq \sup_{v \in \vadcsp}\, \norm[\hsp]{\hat{\Phi}_N(v)-\Phi(v)}.
	\end{align*}
	The term on the right-hand side can be estimated using
	\Cref{prop:uniformboundsoperator}.
\end{remark}

We establish \Cref{prop:uniformboundsoperator} using
\Cref{lem:subexponential_lipschitz},
which is a direct consequence of 
\cite[Thm.\ 5.1 on p.\ 26]{Buldygin2000}.
\begin{lemma}
	\label{lem:subexponential_lipschitz}
	Let $\rvv$ be an integrable, nonnegative random variable.
	Suppose that $\rvv-\cE{\rvv}$
	is sub-exponential with parameters $(\tau, \Lambda)$.
	We define $L = \cE{\rvv}$.
	Let $\ell = L^2/(2\tau^2)$
	if $L \leq \Lambda \tau^2$
	and $\ell = \Lambda L/2$  otherwise.
	If $\rvv_1, \rvv_2, \ldots$ are independent
	and each $\rvv_i : \Theta \to \real$ has the same distribution as
	$\rvv$, then for all $N \in \naturals$, 
	$\Prob{(1/N)\sum_{i=1}^N \rvv_i \geq 2 L}
	\leq \exp(-N\ell)$.
\end{lemma}
\begin{proof}
	Since $\rvv_i-\cE{\rvv}$ are independent 
	and  sub-exponential with parameters $(\tau, \Lambda)$, 
	we have 
	(see \cite[Thm.\ 5.1 on p.\ 26]{Buldygin2000}),
	\begin{align*}
	\Prob[\Big]{\frac{1}{N}\sum_{i=1}^{N} \rvv_i
		- \cE{\rvv} \geq \varepsilon}
	\leq 
	\begin{cases}
	\exp(-N\varepsilon^2/(2\tau^2)) & 
	\tif \quad 0 < \varepsilon \leq \Lambda \tau^2,\\
	\exp(-N\Lambda \varepsilon/2) 	&
	\tif \quad \varepsilon > \Lambda\tau^2.
	\end{cases}
	\end{align*}
	Choosing $\varepsilon = \cE{\rvv}$
	and using  the definition
	of $\ell$, we obtain the assertion.
\end{proof}

\begin{proof}[{Proof of \Cref{prop:uniformboundsoperator}}]
	The proof is inspired by those
	of \cite[Thms.\ 9.84 and 9.86]{Shapiro2021}.
	The Lipschitz property of $\Phi$ is provided by
	\Cref{lem:expecationgcontinuous}. 
	Fix $\nu > 0$.
	We define
	$K = \mathcal{N}(\nu; \vadcsp, \norm[\hsp]{\cdot}) < \infty$.
	By assumption, there exist
	$v_1, \ldots, v_K \in \hsp$
	such that for each $v \in \vadcsp$, we have
	$\norm[\hsp]{v-v_{k(v)}} \leq \nu$, 
	where $k(v) = \argmin_{1\leq k \leq K}\, 
	\norm[\hsp]{v-v_k}$.
	Furthermore, 
	$G(u,\xi)-\cE{G(u,\xi)}$ has zero mean
	for each $u \in \adhsp$.
	Using \Cref{lem:expecationgcontinuous}, we find that
	\begin{align*}
	\begin{aligned}
	\norm[\hsp]{\hat{\Phi}_N(v)-\Phi(v)}
	&\leq 	\norm[\hsp]{\hat{\Phi}_N(v)-\hat{\Phi}_N(v_{k(v)})}
	+ \norm[\hsp]{\hat{\Phi}_N(v_{k(v)})-\Phi(v_{k(v)})}
	\\
	& \quad + \norm[\hsp]{\Phi(v_{k(v)})-\Phi(v)}
	\\
	&\leq \hat{M}_N \norm[\hsp]{v-v_{k(v)}}
	+ \norm[\hsp]{\hat{\Phi}_N(v_{k(v)})-\Phi(v_{k(v)})}
	+ L\norm[\hsp]{v-v_{k(v)}}
	\\
	&\leq \hat{M}_N \nu
	+ \norm[\hsp]{\hat{\Phi}_N(v_{k(v)})-\Phi(v_{k(v)})}
	+ L \nu.
	\end{aligned}
	\end{align*}
	Defining $u_k = \prox{\varphi/\alpha}{v_k}$,
	we obtain $u_k \in \adhsp$ and
	\begin{align*}
	\begin{aligned}
	\sup_{v \in \vadcsp}\, \norm[\hsp]{\hat{\Phi}_N(v)-\Phi(v)}
	&\leq \hat{M}_N \nu
	+ \max_{1\leq k \leq K}\, 
	\norm[\hsp]{\hat{G}_N(u_k)- \cE{G(u_k,\xi)}}
	+ L \nu.
	\end{aligned}
	\end{align*}
	Taking expectations, 
	using  $\cE{\hat{M}_N} = L$
	and utilizing \Cref{prop:maxmean}, we obtain
	\eqref{eq:expsup}. 
	
	We fix $\varepsilon > 0$
	and choose $\nu = \varepsilon/(4L)$
	as in \cite[p.\ 471]{Shapiro2021}.
	\Cref{prop:maxmean} ensures 
	\begin{align*}
	\Prob{\max_{1\leq k \leq K}\, 
		\norm[\hsp]{\hat{G}_N(u_k)- \cE{G(u_k,\xi)}}
		\geq \varepsilon/4
	}
	\leq 2 K \eu^{-N\varepsilon^2 /(48\tau_G^2)}.
	\end{align*}
	 \Cref{lem:subexponential_lipschitz}
	yields $\Prob{\hat{M}_N \geq 2 L} \leq \exp(-N\ell)$.
	If $\hat{M}_N < 2L$, then
	$(\hat{M}_N+ L)\nu < 3L\nu$.
	Now the union bound implies the first tail bound.
	If $M(\xi) \leq L$, then
	$\Prob{\hat{M}_N \geq 2 L} = 0$
	and we obtain \eqref{eq:probsup'}. 
\end{proof}

\section*{Acknowledgments}
JM thanks Professor Alexander Shapiro for valuable
discussions about the SAA approach.
We thank the anonymous referees for their helpful comments.

\bibliographystyle{siamplain}
\bibliography{JMilz_MUlbrich_2022_SAA4PDE_Complexity_v3.bbl}
\end{document}